\documentclass[journal,10pt,twocolumn]{IEEEtran}  % Comment this line out
\IEEEoverridecommandlockouts                              % This command is only
\title{\LARGE \bf
Geometric Convergence of Gradient Play Algorithms \\ for Distributed Nash Equilibrium Seeking}

\author{Tatiana Tatarenko, Wei Shi and Angelia Nedi\'c% <-this % stops a space
\thanks{The authors are with School of Electrical, Computer and Energy Engineering,
        Arizona State University, USA. 
        }
\thanks{A. Nedi\'c is also affiliated with Moscow Institute of Physics and Technology (MIPT), Dolgoprudny, Russia.}
\thanks{The work has been partially supported by Office of Naval Research grant no. N00014-12-1-0998.}
}
\usepackage{amsmath, amssymb, epsfig, graphicx, bm}
\usepackage{cases}
\usepackage{float, subfigure}
\usepackage{color}
\usepackage{amsthm}
\usepackage{amsfonts}
\usepackage{psfrag}
\usepackage{cite, url}
\usepackage{breqn}
\usepackage[all,cmtip]{xy}
\usepackage{color,hyperref}
\definecolor{darkblue}{rgb}{0,0,1}
\hypersetup{colorlinks,breaklinks,
linkcolor=darkblue,urlcolor=darkblue,anchorcolor=darkblue,citecolor=darkblue}

\usepackage[normalem]{ulem} % to use \sout
\newcommand{\spa}[1]{{\mathrm{span}\{#1\}}}
\newcommand{\nul}[1]{{\mathrm{null}\{#1\}}}
\newcommand{\pro}[2]{{\mathcal{P}_{#1}\left\{{#2}\right\}}}
\newcommand{\smax}[1]{{\sigma_{\max}\{#1\}}}

\newcommand{\lmax}[1]{{\lambda_{\max}\{#1\}}}

\newcommand{\lminnz}[1]{{\tilde{\lambda}_{\min}\{#1\}}}

\newtheorem{theorem}{Theorem}
\newtheorem{definition}{Definition}
\newtheorem{proposition}{Proposition}
\newtheorem{lemma}{Lemma}
\newtheorem{remark}{Remark}

\newtheorem{assumption}{Assumption}

\newcommand{\R}{{\mathbb{R}}}
\newcommand{\bx}{{\mathbf{x}}}
\newcommand{\bc}{{\mathbf{c}}}
\newcommand{\bn}{{\mathbf{n}}}
\newcommand{\by}{{\mathbf{y}}}

\newcommand{\Lam}{{\Lambda}}
\newcommand{\Diag}{{\mathrm{Diag}}}
\newcommand{\di}{{\mathrm{diag}}}

\newcommand{\dJi}[1]{{\nabla_{#1} J_{#1}}}
\newcommand{\sd}{{\tilde{\nabla}}}
\newcommand{\mura}{{\mu_{r,\mathbf{F}_\mathrm{a}}}}

\newcommand{\E}{{\mathcal{E}}}

\newcommand{\A}{{\mathcal{A}}}
\newcommand{\I}{{\mathcal{I}}}
\newcommand{\Gra}{{\mathcal{G}}}

\newcommand{\tx}{{\tilde{x}}}
\newcommand{\ty}{{\tilde{y}}}
\newcommand{\ind}{{\mathcal{I}}}
\newcommand{\Om}{{\Omega}}
\newcommand{\Oma}{{\Omega_\mathrm{a}}}

\newcommand{\N}{{\mathcal{N}}}

\newcommand{\bF}{{\mathbf{F}}}
\newcommand{\tbF}{{\tilde{\mathbf{F}}}}
\newcommand{\bFa}{{\mathbf{F}_\mathrm{a}}}
\newcommand{\bG}{{\mathbf{G}}}

\newcommand{\Fro}{{\mathrm{Fro}}}

\newcommand{\LF}{{L_{{\bF}}}}
\newcommand{\LLF}{{L_{{\Lam\bF}}}}

\newcommand{\muFa}{\mu_{{\bF_{\mathrm{a}}}}}
\newcommand{\muF}{\mu_{{\bF}}}
\newcommand{\LFa}{{L_{{\bF_{\mathrm{a}}}}}}

\newcommand{\trace}{{\mathrm{trace}}}
\newcommand{\T}{{\mathrm{T}}}

\newcommand{\Lap}{\textbf{{\L}}}

\begin{document}
%\abovedisplayskip = 0pt
%\belowdisplayskip = 0pt
%\abovedisplayshortskip = 0pt
%\belowdisplayshortskip= 0pt

\maketitle
\thispagestyle{empty}
\pagestyle{empty}

%%%%%%%%%%%%%%%%%%%%%%%%%%%%%%%%%%%%%%%%%%%%%%%%%%%%%%%%%%%%%%%%%%%%%%%%%%%%%%%%
\begin{abstract}                          % Abstract of not more than 200 words.
We study distributed algorithms for seeking a Nash equilibrium in a class of non-cooperative convex games with strongly monotone mappings. Each player has access to her own smooth local cost function and can communicate to her neighbors in some undirected graph. To deal with fast distributed learning of Nash equilibria under such settings, we introduce a so called augmented game mapping and provide conditions under which this mapping is strongly monotone. We consider a distributed gradient play algorithm for determining a Nash equilibrium (GRANE). The algorithm involves every player performing a gradient step to minimize her own cost function while sharing and retrieving information locally among her neighbors in the network. Using the reformulation of the Nash equilibrium problem based on the strong monotone augmented game mapping, we prove the convergence of this algorithm to a Nash equilibrium with a geometric rate. Further, we introduce the Nesterov type acceleration for the gradient play algorithm. We 
demonstrate that, similarly to the accelerated algorithms in centralized optimization and variational inequality problems, our accelerated algorithm outperforms GRANE in the convergence rate. Moreover, to relax assumptions required to guarantee the strongly monotone augmented mapping, we analyze the restricted strongly monotone property of this mapping and prove geometric convergence of the distributed gradient play under milder assumptions.
\end{abstract}

%%%%%%%%%%%%%%%%%%%%%%%%%%%%%%%%%%%%%%%%%%%%%%%%%%%%%%%%%%%%%%%%%%%%%%%%%%%%%%%%
\section{Introduction}
Game theory provides a framework to deal with optimization problems arising in multi-agent systems, where objective functions are coupled through decision variables of all agents in a system. The applications of game theoretic approach include, for example, electricity markets, power systems,  flow control problems, and communication networks \cite{Alpcan2005, BasharSG, Scutaricdma}. Nash equilibria in games characterize desirable and stable solutions to the corresponding multi-agent optimization problems. The focus of our paper is on discrete-time algorithms applicable to fast distributed Nash equilibrium seeking in a class of non-cooperative games with strongly monotone mappings.

There is a large body of work on Nash equilibrium computation in non-cooperative games. Each approach is based on information available to agents in a systems and takes into account some structural properties of agents' cost functions. For example, in the case of the convex \textit{potential game} structure, a central optimization problem can be formulated whose minimizers are Nash equilibria in the game. To compute the minima of the potential function, some distributed communication-based algorithms can be set up \cite{Li2013}. Moreover, in absence of communication, so called \emph{oracle-based} algorithms can be applied to Nash equilibrium seeking in potential games. In systems with oracle information each agent can calculate her current output given any action from her action set. Various versions of the logit dynamics have been presented to compute Nash equilibria in such setting in both discrete action \cite{MardenLLL, Tat_ACC14} and continuous action \cite{Leslie, TatTCNS} games. In some practical
situations agents do not have access to functional forms of their objectives. Rather, each agent (player) can only observe obtained payoffs and be aware of her own actions. Such information structure is usually referred to as \textit{payoff-based}. A payoff-based Nash equilibrium learning  is proposed in \cite{MardRev, COVEr} for discrete action potential games and in \cite{Frihauf2012, Stankovic2012,  TatKamTAC2018} for continuous action convex games.

Many works study convex games where players can exchange their local information with neighbors in some undirected connected communication graph.
Distributed algorithms are proposed for \textit{aggregative games} \cite{Koshal2012, paccagnan2016aggregative}. Communication protocols are applied to general convex games with some convergence guarantees \cite{salehisadaghiani2016distributed, ADMM_WS2017}. The work \cite{salehisadaghiani2016distributed} studies distributed Nash equilibrium seeking of general non-cooperative games and proposes a gradient based gossip algorithm. Under strict convexity, Lipschitz continuity and bounded gradient assumptions, this algorithm converges almost surely to the Nash equilibrium, given a diminishing step size.  Under further assumption of strong convexity, with some constant step size $\alpha$, the algorithm converges to an $O(\alpha)$ neighborhood of the Nash equilibrium in average. The work \cite{ADMM_WS2017} develops an algorithm within the framework of inexact-ADMM and proves its convergence to the Nash equilibrium with the rate $o(1/k)$ under cocoercivity of the game mapping.

All aforementioned works focus on convergence guarantees in different game settings and do not aim to provide fast distributed algorithms with a geometric convergence rate. The contributions of this paper are as follows. We provide an equivalent condition for Nash equilibrium states based on so called \emph{augmented game mapping} which takes into account agents' local information that is exchanged over some communication graph. Further, similarly to majority of research on distributed Nash equilibrium seeking, we consider undirected connected communication graphs and figure out under which assumptions the augmented game mapping is strongly monotone and Lipschitz continuous. Next, we present a distributed gradient play algorithm for determining a Nash equilibrium (GRANE). We demonstrate that GRANE converges to the Nash equilibrium with a geometric rate, given strongly monotone and Lipschitz continuous augmented game mapping. Moreover, using the results for strongly monotone variational inequalities presented 
in \cite{Nesterov}, we introduce the Nesterov type acceleration for the gradient play algorithm. We show that the corresponding accelerated gradient play algorithm (Acc-GRANE) outperforms GRANE in the convergence rate. However, the assumptions under which the augmented game mapping is strongly monotone restrict the class of games where Acc-GRANE can be applied. To rectify this issue we consider a relaxed assumption implying the restricted strongly monotone game mapping. Although Acc-GRANE is not applicable in this case, we demonstrate that GRANE converges to the Nash equilibrium with a geometric rate under this milder assumption. A preliminary and brief version of our work will appear in Conference 
on Decision and Control 2018 \cite{Cdc2018_TatShiNed}. In the current paper, we extend our past work in the following ways. Firstly, we provide all proofs omitted in 
the conference paper. Secondly, we relax the assumption for the augmented mapping to be able to consider a broader class of games 
where geometric convergence to a Nash equivilibrium takes place. Finally, we discuss how communication topology affects the convergence 
rate of the distributed algorithms.

This paper is organized as follows. In Section \ref{sec:problem},  we set up the game under consideration. In Section \ref{sec:eqcond}, we define the augmented game mapping, analyze its properties, and provide an equivalent condition for Nash equilibrium states based on this mapping. Section \ref{sec:algorithms} develops the algorithms with geometric convergence to a Nash equilibrium. We provide a numerical case study in Section \ref{sec:sim}. In Section \ref{sec:conclusion}, we summarize the result and discuss future work.

\textbf{Notations.}
The set $\{1,\ldots,n\}$ is denoted by $[n]$.
For any function $f:K\to\R$, $K\subseteq\R^n$, $\nabla_i f(x) = \frac{\partial f(x)}{\partial x_i}$ is the partial derivative taken in respect to the $i$th coordinate of the vector variable $x\in\R^n$.
For any real vector space $\tilde E$ its dual space is denoted by $\tilde E^*$ and the inner product is denoted by $\langle u,v \rangle$, $u\in\tilde E^*$, $v\in \tilde E$. Some operator $B:\tilde E\to\tilde E^*$ is positive definite if $\langle Bv,v \rangle>0$ for all $v\in\tilde E\setminus \{0\}$. Some operator $B:\tilde E\to\tilde E^*$ is self-adjoint if $\langle Bv,v' \rangle = \langle Bv',v \rangle$ for all $v',v\in\tilde E$. 
Given a positive definite and self-adjoint operator $B$, we define the Euclidean norm on $\tilde E$ as $\|v\| = \langle Bv,v \rangle^{1/2}$.
Any mapping $g:\tilde E\to \tilde E^*$ is said to be \emph{strongly monotone with the constant} $\mu>0$ on $Q\subseteq \tilde E$, if $\langle g(u)-g(v), u - v \rangle\ge\mu\|u - v\|^2$ for any $u,v\in Q$, where $\|\cdot\|$ is the corresponding norm in $\tilde E$.
We consider real vector space $E$, which is either space of real vectors $E = E^* = \R^n$ or the space of real matrices $E = E^* = \R^{n\times n}$. In the case $E = \R^{n\times n}$ the inner product $\langle u,v \rangle \triangleq \sqrt{\trace(u^Tv)}$ is the Frobenius inner product on $\R^{n\times n}$.
In the case $E = \R^n$ we use $\|\cdot\|_2$ to denote the Euclidean norm induced by the standard dot product in $\R^n$, whereas  in the case $E = \R^{n\times n}$ we use $\|\cdot\|_{\Fro}$ to denote the Frobenius norm induced by the Frobenius inner product i.e. $\|v\|_{\Fro} \triangleq \sqrt{\trace(v^Tv)}$. 
We use $\pro{\Om}{v}$ to denote the projection of $v\in E$ to a set $\Om\subseteq E$.
The largest singular value of a matrix $A$ is denoted by
$\smax{A}$. The smallest \emph{nonzero} eigenvalue of a positive semidefinite matrix $A\not=0$ is denoted by
$\lminnz{A}$, which is strictly positive. For a matrix $A\in\R^{m\times n}$,
 $\nul{A}\triangleq\left\{x\in\R^n\big|Ax=0\right\}$ is the null space of
 $A$ and $\spa{A}\triangleq \left\{y\in\R^m\big|y=Ax,\forall
x\in\R^n\right\}$ is the linear span of all the columns of $A$.
For any matrix $A\in\R^{n\times n}$ we use $\di(A)$ to denote its diagonal vector, i.e. $\di(A) = (a_{11},\ldots, a_{nn})$.
For any vector $a\in\R^n$ we  use $\Diag(a)$ to denote the diagonal matrix with the vector $a$ on its diagonal.
We call a matrix $A$ \emph{consensual}, if it has equal row vectors.
%A column vector $v$ is in $\Om\subseteq\R^n$ if $\forall\ i$ the $i$-th entry of $v$ is in $\Om_i$; A matrix $M$ is in $\Om^\ominus\subseteq\R^{n\times n}$ if $\forall\ i$ the $i$-th diagonal entry of $M$ is in $\Om_i$ and all the other entries are null; also  $\Om_{(i)}^\ominus=\R\times\cdots\times\R\times\Om_i\times\R\times\cdots\times\R$ where $\Om_i$ is placed at the $i$-th entry in the multiple Cartesian products.
%Let $\nabla$ be the full gradient operator, $\nabla_i$ be the partial gradient operator with regard to the $i$-th variable, and $\sd$ be the subgradient operator. A subgradient of a function $f$ at a point $x$ gives an element in the subdifferential set $\partial f(x)$. See \cite{Rockafellar1979,Benoist2005} for the definitions of subdifferential and subgradient. For any projection $y=\pro{\Om}{v}=\arg\min\limits_{x\in\Om} \frac{1}{2\alpha_i}\|x-v\|_\Fro^2$ with arbitrary $v$ and bounded $\alpha_i>0$, a subgradient of $\I_{\Om}$ at the point $\pro{\Om}{v}$ that we can use is $\sd \I_{\Om}(\pro{\Om}{v})=-\frac{1}{\alpha_i}(\pro{\Om}{v}-v)\in\partial \I_{\Om}(\pro{\Om}{v})$. By the arbitrariness of $\alpha_i$, we can see that any vector will be a subgradient of $\I_{\Om}$ at the point $\pro{\Om}{v}$ if the vector is in parallel and has the same direction with the ray $(\overrightarrow{\pro{\Om}{v},v})$. Another implication is that, for an indicator function of a set, only subgradients at points in the set are
Assuming that $\Om\subseteq\R^n$,  we define the indicator function $\I_{\Om}(x)$ of the set $\Om$ such that $\I_{\Om}(x)=0$, if $x\in\Om$, and $\I_{\Om}(x)=+\infty$, otherwise.
%for any point $y\in\Om$, we can actually define the subdifferential set $\partial \I_{\Om}(y)=\left\{-\frac{1}{\alpha}(y-v)\big|y=\pro{\Om}{v},v\in\R^n,\alpha>0\right\}$.
%
%$$\di, \quad \Diag$$
%
%
%Finally, for given matrix $A$ and positive semidefinite matrix $\CG$,
%we define the $\CG$-matrix norm $\|A\|_\CG\triangleq\sqrt{\trace(A^\T
%\CG A)}$. The largest singular value of a matrix $A$ is denoted as
%$\smax{A}$. The $i$-th largest eigenvalue by modulus of a matrix $A$ is denoted as $\lam{i}{A}$. Specially, the largest and smallest eigenvalues by modulus of a
%matrix $A$ are denoted as $\lmax{A}$ and
%$\lmin{A}$, respectively; the smallest \emph{nonzero} eigenvalue by modulus of a positive semidefinite matrix $A\not=0$ is denoted as
%$\lminnz{A}$, which is strictly positive. For a matrix $A\in\R^{m\times n}$,
% $\nul{A}\triangleq\left\{x\in\R^n\big|Ax=0\right\}$ is the null space of
% $A$ and $\spa{A}\triangleq \left\{y\in\R^m\big|y=Ax,\forall
%x\in\R^n\right\}$ is the linear span of all the columns of $A$.

\section{Nash Equilibrium Seeking}\label{sec:problem}
\subsection{Problem Formulation}
We consider a non-cooperative game between $n$ players. Let $J_i$ and $\Om_i\subseteq \R$\footnote{All results below are applicable for games with different dimensions $\{d_i\}$ of the action sets $\{\Om_i\}$. The one-dimensional case is considered for the sake of notation simplicity.} denote respectively the cost function and the feasible action set of the player $i$. We denote the joint action set by $\Om = \Om_1\times\ldots\times\Om_n$. Each function $J_i(x_i,x_{-i})$, $i\in[n]$, depends on $x_i$ and $x_{-i}$, where $x_i\in\Om_i$ is the action of the player $i$ and $x_{-i}\in\Om_{-i}=\Om_1\times\ldots\times\Om_{i-1}\times\Om_{i+1}\times\Om_n$ denotes the joint action of all players except for the player $i$. We assume that the players can interact over an undirected communication graph $\Gra([n],\A)$. The set of nodes is the set of the player $[n]$ and the set of undirected arcs $\A$ is such that $(i,j)\in\A$ if and only if $(j,i)\in\A$, i.e. there is an undirected communication link between $i$ to $j$. 
Thus, some information (message) can be passed from the player $i$ to the player $j$ and vice versa. For each player $i$ the set $\N_i$ is the set of neighbors in the graph $\Gra([n],\A)$, namely $\N_{i}\triangleq\{j\in[n]: \, (i,j)\in\A\}$.
Let us denote the game introduced above by $\Gamma(n,\{J_i\},\{\Om_i\},\Gra)$.
We make the following assumptions regarding the game $\Gamma$.

\begin{assumption}\label{assum:convex}
 The game under consideration is \emph{convex}. Namely, for all $i\in[n]$, the set $\Om_i$ is convex and closed, the cost function $J_i(x_i, x_{-i})$ is convex in $x_i$ and continuously differentiable in $x_i$ for each fixed $x_{-i}$.
\end{assumption}

Under Assumption~\ref{assum:convex}, we can define the game mapping.
\begin{definition}\label{def:gamemapping} The \emph{game mapping} $\bF(x):\Om\to\R^n$ is defined as follows:
 \begin{align}\label{eq:gamemapping}
&\bF(x)\triangleq\left[\nabla_1 J_1(x_1,x_{-1}), \ldots, \nabla_n J_n(x_n,x_{-n})\right]^T.
%\bF(x)&\in\R^{n}.
 \end{align}
\end{definition}

\begin{assumption}\label{assum:str_monotone}
 The game mapping $\bF(x)$ is defined on the whole space $\R^n$ and is \emph{strongly monotone} on $\R^n$ with the constant $\mu_{\bF}$.
\end{assumption}
\begin{remark}\label{rem:str_mon}
  Note that Assumption~\ref{assum:str_monotone} above implies strongly convexity of each cost function $J_i(x_i, x_{-i})$ in $x_i$ for any fixed $x_{-i}$ with the constant $\mu_{\bF}$. Thus, the part of Assumption~\ref{assum:convex} holds. However, we consider Assumptions~\ref{assum:convex} and \ref{assum:str_monotone} separately, to be able to emphasize what conditions are required in each of the results presented further in the paper.
\end{remark}

\begin{assumption}\label{assum:Lipschitz}
For every $i\in[n]$ the function $\dJi{i}(x_i,x_{-i})$ is Lipschitz continuous in
$x_i$ on $\Om_i$ for every fixed $x_{-i}\in\R^{n-1}$, that is, given any $x_{-i}\in\R^{n-1}$, for some constant $L_i\ge 0$, we have $\forall\ x_i, y_i\in\Om_i$
\begin{align*}
 |\dJi{i}(x_i,x_{-i})-\dJi{i}(y_i,x_{-i})|&\leq L_i|x_i-y_i|.
\end{align*}
Moreover, for every $i\in[n]$ the function $\dJi{i}(x_i,x_{-i})$ is Lipschitz continuous in $x_{-i}$ on $\R^{n-1}$, for every fixed $x_i\in\Om_i$, that is, given any $x_{i}\in\Om_i$, for some constant $L_{-i}\ge 0$, we have
\begin{align*}
|\dJi{i}(x_i,x_{-i})-\dJi{i}(x_i,y_{-i})|&\leq L_{-i}\|x_{-i}-y_{-i}\|_2,\cr
 \forall\ x_{-i},& y_{-i}\in\R^{n-1}.
\end{align*}
\end{assumption}

Finally, we make the following assumption on the communication graph, which guarantees sufficient information  "mixing" in the network.
\begin{assumption}\label{assum:connected}
The underlying undirected communication graph $\Gra([n],\A)$ is connected. The associated symmetric non-negative mixing matrix $W=[w_{ij}]\in\R^{n\times n}$ defines the weights on the undirected arcs such that $w_{ij}>0$ if and only if $(i,j)\in\A$ and $\sum_{i=1}^{n}w_{ij} = 1$, $\forall j\in[n]$.
\end{assumption}
\begin{remark}
  Note that there are some simple strategies for generating symmetric mixing matrices for which Assumption~\ref{assum:connected} holds (see Section 2.4 in \cite{Shi2014} for summarized strategies).
\end{remark}
Assumption~\ref{assum:connected} implies the following properties of the mixing matrix $W$ (see \cite{brouwer12}):
\begin{subequations}
\label{eq:Wprop}
\begin{align}
   &\sum_{i=1}^{n}w_{ij} = \sum_{j=1}^{n}w_{ij} = 1,\ \forall i,j\in[n] \label{eq:dstoch}\\
   &W \mbox{ is symmetric, }\, I-W\succcurlyeq0; \label{eq:norm}\\
   & \nul{I-W}=\spa{\mathbf{1}}.\label{eq:null}
\end{align}
\end{subequations}

One of the stable solutions in any game $\Gamma$ corresponds to a Nash equilibrium defined below.
\begin{definition}\label{def:NE}
 A vector $x^*=[x_1^*,x_2^*,\cdots, x_n^*]^T\in\Om$ is a \emph{Nash equilibrium} if for any $i\in[n]$ and $x_i\in \Om_i$
 $$J_i(x_i^*,x_{-i}^*)\le J_i(x_{i},x_{-i}^*).$$
 \end{definition}
In this work, we are interested in \emph{distributed seeking of a Nash equilibrum} in a game $\Gamma(n,\{J_i\},\{\Om_i\},\Gra)$ for which Assumptions~\ref{assum:convex}-\ref{assum:connected} hold.

\subsection{Existence and Uniqueness of the Nash Equilibrium}
In this subsection, we demonstrate the existence and uniqueness of the Nash equilibrium for $\Gamma(n,\{J_i\},\{\Om_i\},\Gra)$ under Assumptions~\ref{assum:convex} and \ref{assum:str_monotone}. For this purpose we recall the results connecting Nash equilibria and solutions of variational inequalities \cite{FaccPang1}.

\begin{definition}\label{def:VI}
For a set $Q \subseteq \R^d$ and a mapping $g$: $Q \to \R^d$ the \emph{variational inequality} problem $VI(Q,g)$ is formulated as follows:
\[\mbox{Find }q^* \in Q: \quad \langle g(q^*), q-q^*\rangle \ge 0 \mbox{ for all $q \in Q$}.\] 
\emph{The set of solutions} to the variational inequality problem $VI(Q,g)$ is denoted by $SOL(Q,g)$. 
\end{definition}

The following theorem is the well-known result on the connection between Nash equilibria in games and solutions of a definite variational inequality (see Corollary 1.4.2 in \cite{FaccPang1}).

\begin{theorem}\label{th:VINE}
 Consider a non-cooperative game $\Gamma$. Suppose that the action sets of the players $\{\Om_i\}$ are closed and convex, the cost functions $\{J_i(x_i,x_{-i})\}$ are continuously differentiable and convex in $x_i$ for every fixed $x_{-i}$ on the interior of the joint action set $\Om$. Then, a vector $x^*\in \Om$ is a Nash equilibrium for the $\Gamma$ if and only if $x^*\in SOL(\Om,\bF)$, where $\bF$ is the game mapping defined by \eqref{eq:gamemapping}.
\end{theorem}

The following result holds for variational inequalities with strongly monotone mappings (see Proposition 2.3.3 in \cite{FaccPang1}).
\begin{theorem}\label{th:existVI1}
 Given the $VI(Q,g)$, suppose that $Q$ is a closed convex set and the mapping $g$ is continuous and strongly monotone.  Then, the solution set $SOL(Q,g)$ is nonempty and is a singleton.
\end{theorem}

Taking into account Theorems~\ref{th:VINE} and \ref{th:existVI1}, we obtain the following results.
\begin{theorem}\label{th:exist}
   Let $\Gamma(n,\{J_i\},\{\Om_i\},\Gra)$ be a game for which Assumptions~\ref{assum:convex}-\ref{assum:str_monotone} hold. Then, there exists a unique Nash equilibrium in $\Gamma$. Moreover, the Nash equilibrium in $\Gamma$ is the solution of $VI(\Om,\bF)$, where $\bF$ is the game mapping (see \eqref{eq:gamemapping}).
\end{theorem}

Finally, if Assumption~\ref{assum:str_monotone} does not hold, but in addition to Assumption~\ref{assum:convex} all action sets are compact, then, according to Corollary 2.2.5 in \cite{FaccPang1}, we can guarantee only existence of a Nash equilibrium in the game under consideration.
\begin{theorem}\label{th:exist1}
   Let $\Gamma(n,\{J_i\},\{\Om_i\},\Gra)$ be a game for which Assumption~\ref{assum:convex} hold and $\Om_i$ is a compact set for all $i\in[n]$. Then, there exists a Nash equilibrium in $\Gamma$. Moreover, any Nash equilibrium in $\Gamma$ is the solution of $VI(\Om,\bF)$, where $\bF$ is the game mapping (see \eqref{eq:gamemapping}).
\end{theorem}

Thus, if Assumptions~\ref{assum:convex}-\ref{assum:str_monotone} hold, we can guarantee existence and uniqueness of the Nash equilibrium in the game $\Gamma(n,\{J_i\},\{\Om_i\},\Gra)$ under consideration. However, the reformulation of the Nash equilibrium in terms of the solution of the variational inequality  $VI(\Om,\bF)$ does not take into account the distributed setting of the problem presented in the previous subsection. Hence, to let the players learn the unique Nash equilibrium in a distributed way, we need to find an alternative condition which allows us to determine the Nash equilibrium in presence of partial information exchange over the communication graph $\Gra$ with the associated mixing matrix $W$. In the next section, we provide such a condition and present  distributed algorithms for computing the Nash Equilibrium.

\section{Equivalent Condition for Nash Equilibria and Augmented Game Mapping}\label{sec:eqcond}
\subsection{Nash Equilibria in Distributed Settings}
To deal with the partial information available to players and exchanged among them over the communication graph, we assume that each player $i$ maintains a \emph{local variable}
\begin{align}\label{eq:vector}
x_{(i)}=[\tx_{(i)1},\cdots,\tx_{(i)i-1},x_i,\tx_{(i)i+1},\cdots,\tx_{(i)n}]^T\in\R^n,
\end{align}
which is her estimation of the joint action $x=[x_1,x_2,\cdots,x_n]^T\in\Om$.
Here $\tx_{(i)j}\in\R$ is the player $i$'s estimate of $x_j$ and  $\tx_{(i)i}=x_i\in\Om_i$. Also, we compactly denote the estimates of other players' actions by the player $i$ as
\begin{align}\label{eq:vector1}
\tx_{-i}=[\tx_{(i)1},\cdots,\tx_{(i)i-1},\tx_{(i)i+1},\cdots,\tx_{(i)n}]^T\in\R^{n-1},
\end{align}
and the estimates of the player $j$'s action $x_j$ by all players as
$$\tx_{(:)j}=[\tx_{(1)j},\cdots,\tx_{(j-1)j},x_{j},\tx_{(j+1)j},\cdots,\tx_{(n)j}]^T\in\R^n.$$
Thus, we can define the estimation matrix $\bx\in\R^{n\times n}$, where the $i$th row is equal to the estimation vector $x_{(i)}$, $i\in[n]$, namely
$$
  \bx\triangleq\left(
     \begin{array}{ccc}
       \textrm{---}& x_{(1)}^\T & \textrm{---} \\
       \textrm{---}& x_{(2)}^\T & \textrm{---} \\
       &\vdots& \\
       \textrm{---}& x_{(n)}^\T & \textrm{---} \\
     \end{array}
   \right).
$$
Note that the estimation matrix belongs to a subset of the space $\R^{n\times n}$ that we denote by $\Oma$. The set $\Oma$ consists of matrices whose diagonal vectors are from the set of joint action set $\Om$, i.e. $\Oma \triangleq \{\bx\in \R^{n\times n}:\, \di(\bx) \in \Om\}.$
Finally, for any given estimation matrix, we define the diagonal matrix $\tbF(\bx)\in\R^{n\times n}$ with $\tbF(\bx)_{ii}= \nabla_i J_i(x_{(i)})$, $i\in[n]$, namely
%\begin{align}\label{eq:diaggrad}
% \tbF(\bx)\triangleq\left(
%     \begin{array}{cccc}
%        \nabla_1 J_1(x_{(1)}) & 0 & \cdots &0\\
%        0 & \nabla_2 J_2(x_{(2)}) & \cdots &0\\
%        \vdots&\vdots&\ddots&\vdots\\
%        0 & 0 & \cdots &\nabla_n J_n(x_{(n)})\\
%     \end{array}
%   \right).
%\end{align}
\begin{align}\label{eq:diaggrad}
 \tbF(\bx)\triangleq \Diag(\nabla_1 J_1(x_{(1)}),\ldots,\nabla_n J_n(x_{(n)})).
\end{align}

Now we are ready to state the proposition providing a necessary and sufficient condition for a joint action $x^*$ to be a Nash equilibrium in the game  $\Gamma(n,\{J_i\},\{\Om_i\},\Gra)$.
\begin{proposition}\label{prop:opt}
Let us consider the game $\Gamma(n,\{J_i\},\{\Om_i\},\Gra)$ for which Assumptions~\ref{assum:convex} and \ref{assum:connected} hold. Then the
following statements are equivalent
\begin{enumerate}
\item  The vector $x^* = [x_1^*,\ldots,x_n^*]^T$ is a Nash equilibrium in $\Gamma(n,\{J_i\},\{\Om_i\},\Gra)$.
\item There exists an estimation matrix $\bx^* \in\Oma$ with the diagonal vector $x^* = [x_1^*,\ldots,x_n^*]^T\in\Omega$ and corresponding diagonal matrix $\tbF(\bx^*)$ (see~\eqref{eq:diaggrad}) such that for any $\bx \in\Oma$ the following holds:
\begin{align*}
  &\langle(I-W)\bx^* + \Lam\tbF(\bx^*), \bx - \bx^*\rangle \ge 0,
  \end{align*}
  where $\Lam=\Diag(\alpha_1,\ldots,\alpha_n)\in\R^{n\times n}$ is an arbitrary diagonal matrix with $\alpha_i>0$, $\forall i\in[n]$.
\end{enumerate}
\end{proposition}
\begin{proof}
From the definition of Nash equilibrium (see De\-finition~\ref{def:NE}) and Assumption~\ref{assum:convex} we conclude that $x^* = [x_1^*,\ldots,x_n^*]^T$ is a Nash equilibrium in $\Gamma(n,\{J_i\},\{\Om_i\},\Gra)$ if and only if there exists the estimation vector $x_{(i)}^*\in\R^n,\forall i\in[n]$, such that
\begin{subequations}\label{eq:NE_optcond}
\begin{align}
  \langle \dJi{i}(x_{(i)}^*),x_i - x_i^*\rangle & \ge 0, \forall x_i\in\Om_i, \label{eq:NE_optcond1}\\
  x_{(1)}^*=x_{(2)}^*=\ldots&=x_{(n)}^*. \label{eq:NE_optcond2}
\end{align}
\end{subequations}
Now we proceed with showing the equivalence \eqref{eq:NE_optcond} $\Leftrightarrow$ 2).

The implication \eqref{eq:NE_optcond} $\Rightarrow$ 2) holds, if we take $\bx^*$ to be the matrix with each row equal to transpose of $x^*$. Indeed, in this case $(I-W)\bx^* = 0$ (since $\sum_{i=1}^{n}w_{ij} = 1$, $\forall j\in[n]$, due to Assumption~\ref{assum:connected}) and, according to \eqref{eq:NE_optcond}, $\dJi{i}(x^*)(x_i - x_i^*)\ge 0$ for all $i\in[n]$.

To show 2) $\Rightarrow$ \eqref{eq:NE_optcond} we define a matrix $\bx(i,j)$, where its element $\bx(i,j)_{kl}$, $k\ne i$, $l\ne j$, is equal to the corresponding element $\bx^*_{kl}$ of the matrix $\bx^*$, $\bx(i,j)_{ij}$ is an arbitrary real number, if $i\ne j$, and $\bx(i,j)\in \Omega_i$, if $i=j$. Next, let us consider the inequalities
\[\langle(I-W)\bx^* + \Lam\tbF(\bx^*), \bx(i,j) - \bx^*\rangle \ge 0, \ \text{ for all $i,j\in[n]$}.\]
From the inequalities above we can conclude that
\begin{equation}\label{eq:line13}
  (I-W)\bx^* + \Lam(\tbF(\bx^*) + \bG(x^*)) = \mathbf{0},
\end{equation}
where $\bG(x^*) = \Diag(\sd \I_{\Om_1}(x_1),\ldots,\sd \I_{\Om_n}(x_n))$ and $\sd \I_{\Om_i}(x_i)$ is a subgradient of the indicator function $\ind_{\Om}(x_i)$ of the set $\Omega_i$.

Recall that the $j$-th column of $\bx^*$ is denoted by $\tx_{(:)j}^*$. For each $\tx_{(:)j}^*,\ j=1,\cdots,n$, discarding the $j$-th equation contained in \eqref{eq:line13} gives
$$[I-W]_{-j}\tx_{(:)j}^*=0 ,\ j=1,\cdots,n.$$
According to \eqref{eq:null}, the matrix $[I-W]_{-j}$ has the full row rank for any $j\in[n]$. Thus, that only vectors $\tx_{(:)j}^*\in\spa{\mathbf{1}},\ j\in[n],$ are solutions to the systems of the linear equations above.

 Therefore, $(I-W)\bx^*=0$ and \eqref{eq:NE_optcond2} holds. Furthermore, we have $\tbF(\bx^*)+\bG(\bx^*)=0$ and, thus, \eqref{eq:NE_optcond1} follows.
\end{proof}

 Note that, regarding conditions for the communication graph, to prove the proposition above we used only assumption on stochastic rows of the matrix $W$ and condition \eqref{eq:null}, which can correspond to any directed strongly connected communication graph. Thus, the equivalent reformulation of Nash equilibria presented in Proposition~\ref{prop:opt} holds also for any convex game $\Gamma(n,\{J_i\},\{\Om_i\},\Gra)$, where the graph $\Gra$ is possibly directed, strongly connected and the mixing matrix $W$ is row stochastic, i.e. $\sum_{i=1}^{n}w_{ij} = 1$, $\forall j\in[n]$.

Moreover, in proof of  Proposition~\ref{prop:opt} we did not use Assumption~\ref{assum:str_monotone} requiring strongly monotone game mapping. However, we need this assumption in the next subsections, where we present distributed algorithms and prove their convergence to the Nash Equilibrium with a geometric rate.

\subsection{Augmented Mapping}
According to Proposition~\ref{prop:opt}, to determine a Nash equilibrium in the game $\Gamma(n,\{J_i\},\{\Om_i\},\Gra)$ under consideration we need find an estimation matrix $\bx^*$ that solves the following variational inequality:
\begin{align}\label{eq:distVI}
  &\langle(I-W)\bx^*+ \Lam\tbF(\bx^*), \bx - \bx^*\rangle \ge 0, \quad \forall \bx\in\Oma.
\end{align}
where $\Lam\in\R^{n\times n}$ is an arbitrary diagonal matrix $\Lam = \Diag(\alpha_1,\ldots,\alpha_n),$ $\alpha_i>0$, $\forall i\in[n]$.
Due to Proposition~\ref{prop:opt}, any solution matrix $\bx^*$ of the variational inequality above is consensual and its diagonal vector, as well as any row, represents a Nash equilibrium in $\Gamma$. Further we call such
estimation matrix $\bx^*$ the \emph{Nash equilibrium matrix}.
Note that \eqref{eq:distVI} is defined in terms of the matrix mapping $\bFa:\R^{n\times n}\rightarrow\R^{n\times n}$:
\begin{align}\label{eq:augmap}
 \bFa(\bx) = (I-W)\bx + \Lam\tbF(\bx).
\end{align}
We refer to $\bFa(\bx)$ above as to the \emph{augmented mapping} of the game $\Gamma(n,\{J_i\},\{\Om_i\},\Gra)$. In contrast to the game mapping $\tbF(\bx)$ (see \eqref{eq:gamemapping}), the augmented mapping allows us to take into account the partial information exchange among the players over the graph $\Gra$. To set up efficient distributed procedures with fast convergence rates, we leverage the results on centralized algorithms for classical variational inequalities with strongly monotone and Lipschitz continuous mappings obtained in \cite{Nesterov}. The summary of these results and corresponding algorithms for distributed settings are presented in the next subsections. We conclude this subsection by formulating conditions under which the augmented mapping is Lipschitz continuous, strongly monotone, or restricted strongly monotone and, thus, the results from \cite{Nesterov} can be applied to the game $\Gamma(n,\{J_i\},\{\Om_i\},\Gra)$.

\begin{lemma}\label{lem:augmap_LC}
 Under Assumption~\ref{assum:Lipschitz}, the augmented mapping $\bFa(\bx)$ of the game $\Gamma(n,\{J_i\},\{\Om_i\},\Gra)$ is Lipschitz continuous on $\Oma$ with the Lipschitz constant $\LLF+\smax{I-W}$, where $\LLF=\max\limits_{i}\left\{\alpha_i\sqrt{L_i^2+L_{-i}^2}\right\}$.
\end{lemma}
\begin{proof}
  By Assumption \ref{assum:Lipschitz}, we see that $\forall x, y\in\R^n$ such that $x_i, y_i\in\Om_i$ and $x_{-i}, y_{-i}\in\R^{n-1}$, we have
\begin{align*}
&    \|\dJi{i}(x_i,x_{-i})-\dJi{i}(y_i,y_{-i})\|_2\cr
& =  \|\dJi{i}(x_i,x_{-i})-\dJi{i}(y_i,x_{-i})\cr
&\qquad\quad +\dJi{i}(y_i,x_{-i})-\dJi{i}(y_i,y_{-i})\|_2\cr
&\leq(\beta\|\dJi{i}(x_i,x_{-i})-\dJi{i}(y_i,x_{-i})\|_2^2\cr
&\qquad\quad +\frac{\beta}{\beta-1}\|\dJi{i}(y_i,x_{-i})-\dJi{i}(y_i,y_{-i})\|_2^2)^{\frac{1}{2}}\cr
&\qquad\qquad \mbox{ (for any $\beta>1$)}\cr
&\leq\left(\beta L_i^2\|x_i-y_i\|_2^2+\frac{\beta}{\beta-1}L_{-i}^2\|x_{-i}-y_{-i}\|_2^2\right)^{\frac{1}{2}}\cr &\qquad\qquad \mbox{ (choose $\beta=1+\frac{L_{-i}^2}{L_i^2}$)}\cr
& =  \sqrt{L_i^2+L_{-i}^2}\|x-y\|_2\triangleq L_{(i)}\|x-y\|_2,
\end{align*}
therefore,
\begin{align}\label{eq:tbF}
  \|\tbF(\bx)-\tbF(\by)\|_\Fro\leq\LF\|\bx-\by\|_\Fro,\ \forall\ \bx,\ \by,
\end{align}
where $\LF=\max\limits_{i}\left\{L_{(i)}\right\}$.

Thus, we have
\begin{align*}
\|\Lam(\tbF(\bx)&-\tbF(\by))+(I-W)(\bx-\by)\|_\Fro\cr
&\leq\left(\LLF+\smax{I-W}\right)\|\bx-\by\|_\Fro,\ \forall\ \bx,\ \by,
\end{align*}
where $\LLF=\max\limits_{i}\left\{\alpha_i L_{(i)}\right\}$.
\end{proof}

\begin{lemma}\label{lem:augmap_strmon}
 Under Assumptions~\ref{assum:str_monotone} and \ref{assum:connected} and given
 \begin{subequations}\label{eq:a}
 \begin{align}
   &a_1 = \tilde\lambda_{\min}\{I-W\}\cr
&\quad-0.5\max_i\{\alpha_i(\sqrt{\mu_{\bF}^2+L_{-i}^2}-\mu_{\bF})\}>0, \label{eq:a1}\\
   &a_2 = \min_i\left\{\frac{\alpha_i}{n}(\mu_{\bF}-L_{-i}\sqrt{n-1})\right\}>0, \label{eq:a2}
 \end{align}
 \end{subequations}
 the augmented mapping $\bFa(\bx)$ of the game $\Gamma(n,\{J_i\},\{\Om_i\},\Gra)$ is strongly monotone with the constant
 \begin{align*}
  \mu_{\bFa} = \min\{a_1,a_2\}.
 \end{align*}
\end{lemma}
\begin{proof}
	(i) For any $\bx,\by\in\Oma$, when $\bx-\by$ is not consensual,  due to Assumption~\ref{assum:str_monotone} (see Remark~\ref{rem:str_mon}), we have
		\begin{align}\label{eq:lemma2_1}
	&\langle\Lam(\tbF(\bx)-\tbF(\by)),\bx-\by\rangle\cr
	&=\sum\limits_{i=1}^n\langle\alpha_i(\nabla_i J_i(x_i,\tx_{-i})-\nabla_i J_i(y_i,\ty_{-i})),x_i-y_i\rangle\cr
	&=\sum\limits_{i=1}^n\langle\alpha_i(\nabla_i J_i(x_i,\tx_{-i})-\nabla_i J_i(y_i,\tx_{-i})),x_i-y_i\rangle\cr
&\qquad+\sum\limits_{i=1}^n\langle\alpha_i(\nabla_i J_i(y_i,\tx_{-i})-\nabla_i J_i(y_i,\ty_{-i})),x_i-y_i\rangle\cr
%	&\geq\sum_{i=1}^n\alpha_i\mu_{\bF}\|x_i-y_i\|_2^2\cr
%&\qquad-\sum\limits_{i=1}^n\alpha_i\|\nabla_i J_i(y_i,\tx_{-i})-\nabla_i J_i(y_i,\ty_{-i}))\|\cdot\|x_i-y_i\|\cr
	&\geq\sum_{i=1}^n\alpha_i\mu_{\bF}(x_i-y_i)^2-\sum\limits_{i=1}^n\alpha_i L_{-i}\|\tx_{-i}-\ty_{-i}\|_2\cdot|x_i-y_i|\cr
	&\geq\sum_{i=1}^n\alpha_i\mu_{\bF}(x_i-y_i)^2\cr
&\quad-\sum\limits_{i=1}^n\{\frac{\alpha_i L_{-i}\beta_i}{2}\|\tx_{-i}-\ty_{-i}\|_2^2+\frac{\alpha_i L_{-i}}{2\beta_i}(x_i-y_i)^2\}\cr
&\qquad \mbox{($\beta_i>0$ $\forall i\in[n]$ are arbitrary)}.
	\end{align}

	Further choosing $\beta_i=\frac{\sqrt{\mu_{\bF}^2+L_{-i}^2}-\mu_{\bF}}{L_{-i}}$ in \eqref{eq:lemma2_1}, we have
	\begin{equation}\label{eq:lemma2_2}
	\begin{array}{rcl}
	& &\langle\Lam(\tbF(\bx)-\tbF(\by)),\bx-\by\rangle\\
%	&=&\sum\limits_{i=1}^n-0.5\alpha_i(\sqrt{\mu_{\bF}^2+L_{-i}^2}-\mu_{\bF})\|x_{(i)}-y_{(i)}\|_2^2\\
	&\geq&-0.5\max_i\{\alpha_i(\sqrt{\mu_{\bF}^2+L_{-i}^2}-\mu_{\bF})\}\|\bx-\by\|_\Fro^2.
	\end{array}
	\end{equation}
	Thus, due to Assumption~\ref{assum:connected} (see \eqref{eq:norm}), we obtain
	\begin{align*}
		 \langle&\bFa(\bx)-\bFa(\by),\bx-\by\rangle\cr
		&\geq\tilde\lambda_{\min}\{I-W\}\|\bx-\by\|_\Fro^2\cr
&\quad-0.5\max_i\{\alpha_i(\sqrt{\mu_{\bF}^2+L_{-i}^2}-\mu_{\bF})\}\|\bx-\by\|_\Fro^2,
	\end{align*}
	Given the condition \eqref{eq:a1} and definition of $\mu_{\bF_a}$, we conclude that
 $$\langle\bFa(\bx)-\bFa(\by),\bx-\by\rangle \ge \mu_{\bF_a}\|\bx-\by\|_\Fro^2.$$
	
	(ii) For any $\bx,\by\in\Oma$, when $\bx-\by$ is consensual, \eqref{eq:lemma2_1} is still valid.  However, $(I-W)(\bx-\by)=0$. Thus,
	\begin{align*}
	&\langle\bFa(\bx)-\bFa(\by),\bx-\by\rangle=\langle\Lam(\tbF(\bx)-\tbF(\by)),\bx-\by\rangle\cr
	&\geq\sum_{i=1}^n\alpha_i\mu_{\bF}(x_i-y_i)^2\cr
&-\sum\limits_{i=1}^n\{\frac{\alpha_i L_{-i}\beta_i}{2}\|\tx_{-i}-\ty_{-i}\|_2^2+\frac{\alpha_i L_{-i}}{2\beta_i}(x_i-y_i)^2\}.
	\end{align*}
	Further, by choosing $\beta_i=\frac{1}{\sqrt{n-1}}$ and taking into account \eqref{eq:a2}, we get
	\begin{align*}
	&\langle\bFa(\bx)-\bFa(\by),\bx-\by\rangle\cr
	&\geq\sum\limits_{i=1}^n\alpha_i(\mu_{\bF}-L_{-i}\sqrt{n-1})\|x_i-y_i\|_2^2\cr
	&\geq\min_i\left\{\frac{\alpha_i}{n}(\mu_{\bF}-L_{-i}\sqrt{n-1})\right\}\|\bx-\by\|_\Fro^2\cr
    &\geq\mu_{\bF_a}\|\bx-\by\|_\Fro^2.
	\end{align*}
\end{proof}
\begin{remark}\label{rem:augm_str_mon}
Lemma~\ref{lem:augmap_strmon} holds for games with $\mu_{\bF}>\sqrt{n-1}L_{-i}$ for all $i$. Intuitively speaking, this means that, the strong monotonicity of the game mapping should be strong enough while for each player $i$, its local objective's dependency on other player's strategies should be less than quadratically or have small enough quadratic terms. To make this easier to imagine, let us consider a class of game with objectives in the form of $J_i(x_i,x_{-i})=f_i(x_i)+l_i(x_{-i})x_i$, where $f_i$'s are twice differentiable and strongly convex while $l_i(x_{-i})$'s are linear in $x_{-i}$. For this special case, the Jacobian of the game mapping $\bF(\bx)$ is always positive definite as long as the strong convexity constant of $f_i(x_i)$'s are large enough because the Jacobian's off-diagonal entries are all bounded. Meanwhile, in this case, the Lipschitz constant $L_{-i}$ is any number that satisfies $|l_i(x_{-i})-l_i(y_{-i})|\leq L_{-i}\|x_{-i}-y_{-i}\|_2$ for all $x_{-i},y_{-i}\in\R^{n-1}$.
This is fulfilled by any Lipschitz continuous function and linear function is obviously one of the members.
\end{remark}

To rectify the issue pointed out in Remark~\ref{rem:augm_str_mon} and, thus,to broaden the class of games under consideration, let us  relax Assumption~\ref{assum:str_monotone}. Namely, instead of the assumption on strongly monotone game mapping $\bF$, let us introduce the following assumption of \emph{restricted strongly monotone} game mapping.
\begin{assumption}\label{assum:restr_str_mon}
  The game mapping $\bF$ is restricted strongly monotone over $\R^n$ in respect to any Nash equilibrium $x^*\in\Om$ with the constant $\mu_r>0$, namely for any $x\in\R^n$
  \[\langle F(x) - F(x^*),x-x^*\rangle\ge \mu_r\|x-x^*\|^2_2.\]
\end{assumption}

Next, we show that under Assumption~\ref{assum:restr_str_mon} there exist settings for which the augmented mapping $\bFa$ is restricted strongly monotone in respect to the Nash equilibrium matrix $\bx^*$ with some positive constant.

\begin{lemma}\label{lem:augmap_restr}
  Under Assumptions~\ref{assum:Lipschitz}, \ref{assum:connected}, and \ref{assum:restr_str_mon}, the augmented mapping $\bFa(\bx)$ of the game $\Gamma(n,\{J_i\},\{\Om_i\},\Gra)$ is restricted strongly monotone over $\Oma$ in respect to any Nash equilibrium  matrix $\bx^*$ with the constant
  \begin{align*}
  \mura &= \min\{b_1,b_2\},\cr
  b_1 &= \min\left\{\alpha\left(\frac{\mu_r}{n} - \LF(\beta^2 + 2\beta)\right),\tilde\lambda_{\min}\{I-W\} \right\},\cr
  b_2 &= \frac{\tilde\lambda_{\min}\{I-W\}}{1+\frac{1}{\beta^2}} - \alpha \LF,
  \end{align*}
  where $\beta >0$ is a positive constant and $\LF=\max\limits_{i}\left\{\alpha\sqrt{L_i^2+L_{-i}^2}\right\}$.
\end{lemma}
\begin{proof}
As the set of all consensual matrices in $\R^{n\times n}$ is a linear subspace of $\R^{n\times n}$, for any estimation matrix $\bx\in\Oma$ we can consider the following linear decomposition:
\[\bx = \bc +\bn, \]
where $\bc\in\Oma$ is some consensual matrix and $\bn$ is some matrix from the subspace orthogonal to the subspace of all consensual matrices on $\Oma$, namely $\langle\bc,\bn\rangle = 0$. Since $\bx^*$ is a consensual matrix, $\langle\bc - \bx^*,\bn\rangle = 0$. Thus, we get
\begin{align}\label{eq:decom}
\|\bx-\bx^*\|^2_{\Fro} = \|\bc-\bx^*\|^2_{\Fro} + \|\bn\|^2_{\Fro}.
\end{align}
Let $c = \di(\bc)\in\R^n$ denote the vector in rows of the matrix $\bc$. Due to Assumption~\ref{assum:restr_str_mon} and given $\alpha_i=\alpha$ for all $i\in[n]$, we obtain
\begin{align}\label{eq:part1}
 \langle&\Lam(\tbF(\bc) - \tbF(\bx^*)), \bc - \bx^*\rangle \cr
& = \sum_{i=1}^{n}\alpha_i(\nabla_i J_i(c)-\nabla_i J_i(x^*))(c_i - x^*_i)\cr
&\ge\alpha\langle\bF(c) -\bF(x^*),c - x^*\rangle\cr
&\ge\alpha\mu_r\|c-x^*\|^2_2 = \alpha\frac{\mu_r}{n}\|\bc - \bx^*\|^2_{\Fro}.
\end{align}
Since $\tbF$ is Lipschitz continuous on $\Oma$ with the constant $\LF=\max\limits_{i}\left\{\alpha\sqrt{L_i^2+L_{-i}^2}\right\}$ (see equation \eqref{eq:tbF} in the proof of Lemma~\ref{lem:augmap_LC}), we conclude that
 \begin{subequations}\label{eq:part}
 \begin{align}
  \langle\Lam(\tbF(\bc) - \tbF(\bx^*)), \bx - \bc\rangle\ge&-\alpha\LF\|\bc-\bx^*\|_{\Fro}\|\bx-\bc\|_{\Fro}\cr
 =&-\alpha\LF\|\bc-\bx^*\|_{\Fro}\|\bn\|_{\Fro}, \label{eq:part2}\\
   \langle\Lam(\tbF(\bx) - \tbF(\bc)), \bc - \bx^*\rangle\ge&-\alpha\LF\|\bx-\bc\|_{\Fro}\|\bc-\bx^*\|_{\Fro}, \label{eq:part3}\\
 \langle\Lam(\tbF(\bx) - \tbF(\bc)), \bx - \bc\rangle\ge&-\alpha\LF\|\bx-\bc\|^2_{\Fro}\cr
  &\qquad\qquad=-\alpha\LF\|\bn\|^2_{\Fro}, \label{eq:part4}
 \end{align}
 \end{subequations}
Next, taking into account inequalities \eqref{eq:part1} and \eqref{eq:part}, we obtain
\begin{align}\label{eq:tbFx}
   &\langle\Lam(\tbF(\bx) - \tbF(\bx^*)), \bx - \bx^*\rangle = \langle\Lam(\tbF(\bc) - \tbF(\bx^*)), \bc - \bx^*\rangle\cr
    & + \langle\Lam(\tbF(\bc) - \tbF(\bx^*)), \bx - \bc\rangle + \langle\Lam(\tbF(\bx) - \tbF(\bc)), \bc - \bx^*\rangle\cr
    &\qquad\qquad\qquad\qquad\qquad +  \langle\Lam(\tbF(\bx) - \tbF(\bc)), \bx - \bc\rangle\cr
    &\ge\alpha\left(\frac{\mu_r}{n}\|\bc - \bx^*\|^2_{\Fro} - \LF\|\bn\|_{\Fro}(2\|\bc-\bx^*\|_{\Fro} + \|\bn\|_{\Fro})\right).
\end{align}
Moreover, due to \eqref{eq:decom},
\begin{align}\label{eq:I-W}
\langle (I-W)&(\bx - \bx^*), \bx - \bx^*\rangle\ge \tilde\lambda_{\min}\{I-W\}\|\bx - \bx^*\|^2_{\Fro}\cr
 &= \tilde\lambda_{\min}\{I-W\}(\|\bc - \bx^*\|^2_{\Fro} + \|\bn\|^2_{\Fro}) \cr
 &\qquad\qquad\qquad= \tilde\lambda_{\min}\{I-W\} \|\bn\|^2_{\Fro},
\end{align}
since $(I-W)(\bc - \bx^*) = 0$ (see \eqref{eq:null}). Summing inequalities \eqref{eq:tbFx} and \eqref{eq:I-W}, we conclude that
\begin{align}\label{eq:Fa1}
 &\langle \bFa(\bx) - \bFa(\bx^*), \bx - \bx^*\rangle \ge \tilde\lambda_{\min}\{I-W\} \|\bn\|^2_{\Fro} \cr
&+ \alpha\left(\frac{\mu_r}{n}\|\bc - \bx^*\|^2_{\Fro} - \LF\|\bn\|_{\Fro}(2\|\bc-\bx^*\|_{\Fro} + \|\bn\|_{\Fro})\right).\cr
 \end{align}
 Let us fix some positive $\beta>0$.

 (i) If $\|\bn\|_{\Fro}\le\beta\|\bc-\bx^*\|_{\Fro}$, then, according to \eqref{eq:Fa1} and \eqref{eq:decom},
 \begin{align}\label{eq:Fa2}
 &\langle \bFa(\bx) - \bFa(\bx^*), \bx - \bx^*\rangle \ge \tilde\lambda_{\min}\{I-W\} \|\bn\|^2_{\Fro} \cr
&\qquad+ \alpha\left(\frac{\mu_r}{n} - \LF(\beta^2 + 2\beta)\right)\|\bc-\bx^*\|^2_{\Fro}\cr
&\qquad\qquad\qquad\qquad\qquad\qquad\ge b_1\|\bx-\bx^*\|^2_{\Fro},
 \end{align}
 where
 \[b_1 = \min\left\{\alpha\left(\frac{\mu_r}{n} - \LF(\beta^2 + 2\beta)\right),\tilde\lambda_{\min}\{I-W\} \right\}.\]

(ii) If $\|\bn\|_{\Fro}\ge\beta\|\bc-\bx^*\|_{\Fro}$, then, according to \eqref{eq:Fa1} and \eqref{eq:decom} and since $\tbF$ is Lipschitz continuous with the constant $\LF$, we obtain
 \begin{align}\label{eq:Fa3}
 &\langle \bFa(\bx) - \bFa(\bx^*), \bx - \bx^*\rangle \ge \tilde\lambda_{\min}\{I-W\} \|\bn\|^2_{\Fro} \cr
&\qquad\qquad\qquad\qquad\qquad\qquad-\alpha \LF\|\bx-\bx^*\|^2_{\Fro}\cr
&\qquad\qquad\qquad\qquad\qquad\quad=-\alpha \LF\|\bx-\bx^*\|^2_{\Fro}\cr
&\qquad\qquad +\tilde\lambda_{\min}\{I-W\} \left(\frac{\beta^2\|\bn\|^2_{\Fro}}{1+\beta^2} +  \frac{\|\bn\|^2_{\Fro}}{1+\beta^2}\right)\cr
&\qquad\qquad\qquad\qquad\qquad\quad\ge-\alpha \LF\|\bx-\bx^*\|^2_{\Fro}\cr
&\qquad\qquad +\tilde\lambda_{\min}\{I-W\} \left(\frac{\beta^2\|\bn\|^2_{\Fro}}{1+\beta^2} +  \frac{\beta^2\|\bc-\bx^*\|^2_{\Fro}}{1+\beta^2}\right)\cr
%& = -\alpha \LF + \frac{\tilde\lambda_{\min}\{I-W\}\beta^2}{1+\beta^2}
&\qquad\qquad\qquad\qquad\qquad\quad = b_2\|\bx-\bx^*\|^2_{\Fro},
 \end{align}
 where
 \[b_2 = \frac{\tilde\lambda_{\min}\{I-W\}}{1+\frac{1}{\beta^2}} - \alpha \LF.\]

 Thus, combining \eqref{eq:Fa2} and \eqref{eq:Fa3}, we conclude that
 \[\mura=\min\{b_1,b_2\}.\]
\end{proof}

\begin{remark}\label{rem:restr}
Given any game $\Gamma(n,\{J_i\},\{\Om_i\},\Gra)$, for which Assumptions~\ref{assum:Lipschitz}, \ref{assum:connected}, and \ref{assum:restr_str_mon} hold, there exist settings of the augmented mapping $\bFa$ with $\Lam=\Diag(\alpha,\cdots,\alpha)$, $\alpha>0$, such that $\bFa$ is restricted strongly monotone over $\Oma$ in respect to any Nash equilibrium  matrix $\bx^*$ with the positive constant $\mura>0$. Indeed, according to Lemma~\ref{lem:augmap_restr}, we can choose $\beta$ such that $b_1 = \min\left\{\alpha\left(\frac{\mu_r}{n} - \LF(\beta^2 + 2\beta)\right),\tilde\lambda_{\min}\{I-W\} \right\}$ is positive. It holds, for example, if
\[\beta^2 + 2\beta = \frac{\mu_r}{2n\LF}.\]
Note that as $\frac{\mu_r}{2n\LF}>0$, there exists some positive $\beta>0$ solving the equation above.
Further, based on the choice of $\beta>0$, we can proceed with settings for $\alpha>0$. We these settings we need to guarantee that $b_2=\frac{\tilde\lambda_{\min}\{I-W\}}{1+\frac{1}{\beta^2}} - \alpha \LF>0$. It holds, for example, if
\[ \alpha = \frac{\tilde\lambda_{\min}\{I-W\}}{2\LF(1+\frac{1}{\beta^2})}. \]
\end{remark}

\section{Gradient Algorithms with Geometric Rates}\label{sec:algorithms}
\subsection{Gradient and Accelerated Gradient Algorithms for Strongly Monotone Variational Inequalities}\label{subsec:Nesterov}
In this subsection we summarize the results presented in \cite{Nesterov} that allow us to set up distributed procedures with geometric convergence to the Nash Equilibrium in games with strongly monotone mappings.

The paper \cite{Nesterov} deals with the following general settings. Let us consider some finite-dimensional real vector space $E$, the dual space $E^*$. The scalar product
of $x\in E$ and $s\in E^*$ is denoted by $\langle s,x\rangle$. The Euclidean norms are defined by a positive definite and self-adjoint operator $B: E\to\E^*$ as follows: 
\begin{align}\label{eq:Boper}
 \|x\| &= \langle Bx,x\rangle^{1/2}, \quad x\in E,\cr
 \|s\|_{*} &=\langle s, B^{-1}s \rangle^{1/2}, \quad s\in E^*.
\end{align}
Let $Q\subseteq E$ be some convex closed subset of $E$ and $g:Q\to E^*$ be some operator on $Q$. The goal is to solve the corresponding variational inequality problem $VI(Q,g)$:
\[\mbox{Find }x^* \in Q: \quad \langle g(x^*), x-x^*\rangle \ge 0 \mbox{ for all $x \in Q$}\footnote{see also Definition~\ref{def:VI} for the special case $E=E^*=\R^d$.}.\]
Let us introduce $\gamma = \frac{L}{\mu}\ge 1$, which is the condition number of the operator $g$.

The standard iterative gradient-type method for solving the inequality above can be formulated as follows:
\begin{align}\label{eq:Nest_GR}
x^0 &= x \in Q,\cr
x^{k+1} &= \pro{Q}{x_k -\lambda B^{-1}g(x_k)},\quad k\ge0,
\end{align}
where $B$ is an operator defining the Euclidean norms in the spaces $E$ and $E^*$ (see \eqref{eq:Boper}).
In the case of strongly monotone and Lipschitz continuous operator $g$ and given an optimal setting of the stepsize $\lambda$, the procedure above converges to the unique solution of $VI(Q,g)$ with a geometric rate.
\begin{theorem}\label{th:geom}
  (\cite{Nesterov}) Let the operator $g$ be continuous, strongly monotone with some constant $\mu>0$, and Lipschitz continuous with some constant $L>0$.
  Let $x^*$ be the unique solution of $VI(Q,g)$, where $Q$ is convex and closed. Then, given the optimal stepsize $\lambda = \frac{\mu}{L^2}$, the following holds for the gradient algorithm \eqref{eq:Nest_GR}:
  \[\|x^k-x^*\|^2\le \exp\left\{-\frac{k}{\gamma^2}\right\}.\]
\end{theorem}
According to Theorem~\ref{th:geom}, the convergence rate of the procedure \eqref{eq:Nest_GR} depends on the squared condition number $\gamma^2$. As $\gamma>1$, the convergence rate would be faster, if the constant factor $\gamma^2$ is replaced by $\gamma$. To obtain this better dependence of the rate on the condition number, namely to accelerate the algorithm \eqref{eq:Nest_GR}, the paper \cite{Nesterov} follows the idea of the Nesterov type acceleration in optimization of strongly convex functions with Lipshitz continuous gradients (see \cite{Nesterov_Lectures}, Section~2.2). For this purpose the following merit function on $Q$ needs to be introduced:
\[f(x) = \sup_{y\in Q}\{\langle g(y),x-y\rangle + 0.5\mu\|x-y\|^2\}.\]
It is shown in \cite{Nesterov} that the function $f$ is strongly convex on $Q$ with the constant $\mu$. Moreover, it is non-negative on $Q$ and is equal to $0$ only at the unique solution of $VI(Q,g)$. Thus, $x^* = SOL(Q,g)$ is equal to the unique minimum of $f$ on $Q$. Further, analogously to the construction of optimization methods with optimal convergence rates presented in Section~2.2 of \cite{Nesterov_Lectures}, the work \cite{Nesterov} provides the estimate sequence of the function $f$. For construction of the estimate sequence and the accelarated algorithm the authors use the sequence of positive weights $\{\lambda^t\}$, $S^k = \sum_{t=0}^{k}\lambda^t$, and the following functions:
\begin{align}\label{eq:psi}
  \psi_y^{\beta} (x) &= \langle g(y), y - x\rangle - 0.5\beta\|x-y\|^2, \quad \mbox{($\beta>0$)},\cr
  \Psi^k(x) &= \sum_{t=0}^{k}\lambda^t\psi_{y^t}^{\mu} (x),
\end{align}
where $y^t$ is some sequence generated by the  corresponding  iterative process.
The accelerated algorithm is based on the definitions above and can be presented by the following iterations:
\begin{align}\label{eq:Nest_accGR}
\lambda^0 &= 1, \, y^0\in Q, \cr
x^k &= \arg \max_{x\in Q} \Psi^k(x),  \cr
y^{k+1} &= \arg \max_{x\in Q} \psi_{x^k}^{L} (x), \cr
\lambda^{k+1}& = \frac{1}{\gamma}\cdot S^k, \quad k\ge 0.
\end{align}
The following result holds for the procedure above.
\begin{theorem}\label{th:accgeom}
  (\cite{Nesterov}, Theorem 3) 
  Let the operator $g$ be continuous, strongly monotone with some constant $\mu>0$, and Lipschitz continuous with some constant $L>0$.
  Let $x^*$ be the unique solution of $VI(Q,g)$, where $Q$ is convex and closed. Let the output of the algorithm \eqref{eq:Nest_accGR} be
  $\tilde y^k = \frac{1}{S^k}\sum_{t=0}^{k}\lambda^t y^t$. Then
  \[0.5\mu\|\tilde y^k-x^*\|^2\le f(y^0)\cdot\gamma^2\exp\left\{-\frac{k}{\gamma+1}\right\}.\]
\end{theorem}
Thus, according to Theorems~\ref{th:geom} and \ref{th:accgeom}, for large values of $\gamma$, the convergence rate of the accelerated procedure \eqref{eq:Nest_accGR} is much better than the convergence rate of the gradient method \eqref{eq:Nest_GR}, while both algorithms have the same implementation complexity.
In the next section we apply the results above to set up and accelerate distributed Nash equilibrium seeking in the case of strongly monotone game mappings.

\subsection{Distributed Algorithms for Nash Equilibrium Seeking}
Recall that in Proposition~\ref{prop:opt} we obtained the following result for the game $\Gamma(n,\{J_i\},\{\Om_i\},\Gra)$:
\begin{align}\label{eq:VIforDNES}
  x^*\mbox{ is a Nash Equilibrium } \Leftrightarrow \, \langle\bFa(\bx^*), &\bx - \bx^*\rangle \ge 0, \, \cr&\forall \bx\in\Oma,
\end{align}
where $\bFa$ is defined in \eqref{eq:augmap} and $\bx^*$ is the Nash equilibrium matrix.
Moreover, it was shown in Lemmas~\ref{lem:augmap_LC} and \ref{lem:augmap_strmon} that under Assumptions~\ref{assum:str_monotone} and \ref{assum:Lipschitz} there are settings for the parameters $\alpha_i$, $i\in[n]$, such that the mapping $\bFa$ is Lipschitz continuous and strongly monotone on $\Oma$ with the constants $\LFa$ and $\muFa$ respectively (see also Remark~\ref{rem:augm_str_mon}). Thus, we can apply the results from \cite{Nesterov} discussed in the preceding section to distributed Nash equilibrium seeking in the game $\Gamma(n,\{J_i\},\{\Om_i\},\Gra)$.

%For example, the vector $x_{(i)}$ at the $k$th iteration  is denoted by $x_{(i)}^{k}$.

We start with adapting the process \eqref{eq:Nest_GR} to $VI(\Oma,\bFa)$ in \eqref{eq:VIforDNES} with the mapping $\bFa(\bx) = (I-W)\bx + \Lam\tbF(\bx)$. Analogous to notations in Subsection~\ref{subsec:Nesterov}, we use the upper index to denote the iteration index.  By writing down the $k$th iteration for each row in the estimation matrix $\bx^k$ we get the following \textbf{Algorithm 1}.
\begin{center}
  {\textbf{Algorithm 1: GRANE}}
  \smallskip

    \begin{tabular}{l}
    \hline
    \emph{  } Set mixing matrix $W$; Choose step size $\lambda>0$;\\
    \emph{  } All players $i\in[n]$ pick arbitrary initial $x_{(i)}^0\in\R^{n}$,\\
    \emph{  } where the $i$th coordinate of $x_{(i)}^0$ is from $\Om_i$;\\
    \emph{  } \textbf{for} $k=0,1,\ldots$, all players $i\in[n]$ do \\
    \emph{~ } $x_{i}^{k+1}=\mathcal{P}_{\Om_i}\{(1-\lambda+\lambda w_{ii})x_i^k+\sum\limits_{j\in\N_{i}}\lambda w_{ij} \tx_{(j)i}^k $\\
    \emph{  } \qquad\qquad\qquad\qquad\qquad\qquad\quad $-\lambda\alpha_i\nabla_i J_i(x_i^k,\tx_{-i}^k)\}$;\\
    \emph{~ } \textbf{for} $l=\{1,\cdots,i-1,i+1,\cdots,n\}$ \\
    \emph{  } \quad$\tx_{(i)l}^{k+1}=\tx_{(i)l}^{k}+\sum\limits_{j\in\N_{i}} \lambda w_{ij} \tx_{(j)l}^k$;\\
    \emph{~ } \textbf{end for};\\
    \emph{  } \textbf{end for}.\\
    \hline
    \end{tabular}
\end{center}
As we can see, to run GRANE, agents start with arbitrary estimations of the joint action, namely, they do not need any prior information on the action sets of other players in the game. As time runs, agents exchange the information about their current estimations with their neighbors over the communication graph with the mixing matrix $W$ and update their action based on local gradient steps. Thus, GRANE is a distributed gradient algorithm. Moreover, due to the relation \eqref{eq:VIforDNES}, Lemmas~\ref{lem:augmap_LC}, \ref{lem:augmap_strmon}, and Theorems~\ref{th:exist}, \ref{th:geom}, the following theorem can be formulated.

\begin{theorem}\label{th:geom_GRANE}
Let us consider the game $\Gamma(n,\{J_i\},\{\Om_i\},\Gra)$ for which Assumptions~\ref{assum:convex}-\ref{assum:connected} hold. Then there exists the unique Nash equilibrium $x^*$ in $\Gamma$. Moreover, given $\alpha_i$, $i\in[n]$, such that $\muFa>0$, and the stepsize $\lambda = \frac{\muFa}{\LFa^2}$, the estimation matrix sequence $\{\bx^k\}$ generated by the algorithm GRANE satisfies
  \[\|\bx^k-\bx^*\|_{\Fro}^2\le \exp\left\{-\frac{k}{\gamma_{\bFa}^2}\right\},\]
  where $\gamma_{\bFa} = \frac{\LFa}{\muFa}$ and $\bx^*$ is the Nash equilibrium matrix.
\end{theorem}

To accelerate GRANE, we apply the updates in \eqref{eq:Nest_accGR} to the augmented game mapping $\bFa$. To set up the accelerated procedure, agents need to update not only the matrix $\bx^k$, but also an auxiliary estimation matrix $\by^k$. The matrix $\by^0$ is initialized in $\Oma$.
%We introduce $\Map^k=(I-W)\by^k+\Lam\bF(\by^k)$ and
Note that, according to the definition of the functions $\psi$ and $\Psi$ in \eqref{eq:psi}, the updates of $x^k$ and $y^{k+1}$ in \eqref{eq:Nest_accGR} are equivalent to
\begin{align*}
  x^k = & \pro{Q}{\frac{1}{S^k}\sum_{i=0}^{k}\left[\lambda^iy^i - \frac{\lambda^i}{\mu}g(y^i)\right]}\cr
  y^{k+1} = &\pro{Q}{x^k - \frac{1}{L}g(x^k)},
\end{align*}
where, as before, $\lambda^0 = 1$, $S^k = \sum_{t=0}^{k}\lambda^t$, $\lambda^{k+1} = \frac{S^k}{\gamma}$, $\mu$ and $L$ are the strong monotonicity and Lipschitz continuity constants of the mapping $g$ respectively, and $\gamma = \frac{L}{\mu}$. Thus, we conclude that the matrices $\bx^k$ and $\by^k$ in the accelerated procedure applied to the Nash equilibrium seeking problem in the game $\Gamma(n,\{J_i\},\{\Om_i\},\Gra)$ are updated as follows:
\begin{equation}
\begin{array}{l}
\by^0\in\Oma, \, \lambda^0 = 1, \, k =0,1,\ldots;\\
S^k = \sum_{t=0}^{k}\lambda^t;\\
\lambda^{k+1} = \frac{S^k}{\gamma_{\bFa}};\\
%\Map^k=(I-W)\by^k+\Lam\bF(\by^k);\\
%\bs^{k}=\bs^{k-1}+u^k(\muFa\by^k-\Map^k);\\
%r^{k}=r^{k-1}+u^k;\\
%u^{k+1}=\frac{\muFa r^{k}}{\LFa};\\
%\bq^{k}=\bq^{k-1}+u^k\by^k;\\
\bx^{k}=\pro{\Oma}{\frac{1}{S^k}\sum_{i=0}^{k}\left[\lambda^t\by^t - \frac{\lambda^t}{\muFa}\bFa(\by^t)\right]};\\
%\bz^{k}=(I-W)\bx^k+\Lam\bF(\bx^k);\\
\by^{k+1}=\pro{\Oma}{\bx^{k}-\frac{1}{\LFa}\bFa(\bx^k)}.\\
%\text{Output } \frac{\bq^{k}}{r^{k}};
\end{array}
\end{equation}
Analogously to \eqref{eq:vector} and \eqref{eq:vector1}, we define vectors $y_{(i)}$ and $\ty_{-i}$ in respect to the matrix $\by\in\Oma$.
Hence, the distributed accelerated process can be presented by \textbf{Algorithm 2} below.

\begin{center}
	{\textbf{Algorithm 2: Acc-GRANE}}
	\smallskip
	
	\begin{tabular}{l}
		\hline
		\emph{  } Set mixing matrix $W$; $\lambda^0 = 1$;  \\
		\emph{  } All players $i\in[n]$ pick arbitrary initial $y_{(i)}^0\in\R^{n}$,\\
        \emph{  } where the $i$th coordinate of $y_{(i)}^0$ is from $\Om_i$;\\
		\emph{  } \textbf{for} $k=0,1,\ldots$, all players $i\in[n]$ do \\
        \emph{~  } $S^k = \sum_{t=0}^{k}\lambda^t$;\\
        \emph{~  } $\lambda^{k+1} = \frac{S^k}{\gamma_{\bFa}}$;\\
        \emph{~  } $x_{i}^{k}=\mathcal{P}_{\Om_i}\{\frac{1}{S^k}\sum_{t=0}^{k}[\lambda^t y_i^t -\frac{\lambda^t}{\muFa}((1-w_{ii})y_i^t$ \\
        \emph{~  }$ \qquad\qquad\qquad\qquad + \sum\limits_{j\in\N_{i}} w_{ij} \ty_{(j)i}^t - \alpha_i\nabla_i J_i(y_i^t,\ty_{-i}^t))]\}$;\\
		\emph{~~  } \textbf{for} $l=\{1,\cdots,i-1,i+1,\cdots,n\}$ \\
		\emph{  } \qquad\qquad$\tx_{(i)l}^{k}=\frac{1}{S^k}\sum_{t=0}^{k}[\lambda^t\ty_{(i)l}^{t}+\frac{\lambda^t}{\muFa}\sum\limits_{j\in\N_{i}} \lambda w_{ij} \ty_{(j)l}^t]$;\\
        \emph{~~  } \textbf{end for};\\
		\emph{~  } $y_i^{k+1}=\mathcal{P}_{\Om_i}\{x_i^k - \frac{1}{\LFa}[(1-w_{ii})x_i^k+\sum\limits_{j\in\N_{i}} w_{ij} \tx_{(j)i}^k $\\
        \emph{~  } \qquad\qquad\qquad\qquad\qquad\qquad\quad $-\alpha_i\nabla_i J_i(x_i^k,\tx_{-i}^k)]\}$;\\		
		\emph{~~  } \textbf{for} $l=\{1,\cdots,i-1,i+1,\cdots,n\}$ \\
		\emph{~~  } \qquad\qquad$y_{(i)l}^{k+1}=\tx_{(i)l}^{k}+\frac{1}{\LFa}\sum\limits_{j\in\N_{i}}w_{ij} \tx_{(j)l}^k$;\\
		\emph{~~  } \textbf{end for};\\
	%	\emph{  } \qquad $y_{(i)}^{k+1}=\pro{\Om_{(i)}^\ominus}{x_{(i)}^k-\frac{1}{\LFa}z_{(i)}^k}$;\\
		\emph{~  }  Output $\tilde{\by}^k = \frac{1}{S^k}\sum_{t=0}^{k}\lambda^t \by^t$;\\		
		\emph{  } \textbf{end for}.\\
		\hline
	\end{tabular}
\end{center}
Similarly to GRANE, Acc-GRANE does not require players to have any prior information on the action sets of other players in the game. Taking into account \eqref{eq:VIforDNES}, Lemmas~\ref{lem:augmap_LC}, \ref{lem:augmap_strmon}, and Theorems~\ref{th:exist}, \ref{th:accgeom}, we obtain the following theorem.

\begin{theorem}\label{th:geom_Acc-GRANE}
Let us consider the game $\Gamma(n,\{J_i\},\{\Om_i\},\Gra)$ for which Assumptions~\ref{assum:convex}-\ref{assum:connected} hold. Then there exists the unique Nash equilibrium $x^*$ in $\Gamma$. Moreover, given $\alpha_i$, $i\in[n]$, such that $\muFa>0$, the matrix sequence of outputs $\{\tilde{\by}^k\}$ generated by the algorithm Acc-GRANE satisfies
  \[0.5\mu_{\bFa}\|\tilde{\by}^k-\bx^*\|_{\Fro}^2\le F(\by^0)\cdot\gamma_{\bFa}^2\exp\left\{-\frac{k}{\gamma_{\bFa}+1}\right\},\]
  where $\gamma_{\bFa} = \frac{\LFa}{\muFa}$, $\bx^*$ is the Nash equilibrium matrix, and
  $F(\by^0) = \sup_{\by\in \Oma}\{\langle \bFa(\by),\by^0-\by\rangle + 0.5\muFa\|\by^0-\by\|_{\Fro}^2\}.$
\end{theorem}

As it has been pointed out in Remark~\ref{rem:augm_str_mon}, the results of Theorems~\ref{th:geom_GRANE} and \ref{th:geom_Acc-GRANE} apply only to the class of games with strongly monotone mappings that additionally satisfy the condition $\mu_{\bF}>\sqrt{n-1}L_{-i}$ for all $i$, implying a restrictive structure of cost functions. To set up the GRANE for a broader class of games, we use Lemma~\ref{lem:augmap_restr} to get the following result.

\begin{theorem}\label{th:geom_GRANEr}
  Let us consider the game $\Gamma(n,\{J_i\},\{\Om_i\},\Gra)$ for which Assumptions~\ref{assum:convex}, \ref{assum:Lipschitz}, \ref{assum:connected}, and \ref{assum:restr_str_mon} hold. Let each action set $\Om_i$, $i\in[n]$, be compact. Then there exists the unique Nash equilibrium $x^*$ in $\Gamma$. Moreover, given $\alpha_i=\alpha>0$, $i\in[n]$, such that $\mura>0$ in Lemma~\ref{lem:augmap_restr}, and the stepsize $\lambda = \frac{\mura}{\LFa^2}$, the estimation matrix sequence $\{\bx^k\}$ generated by the algorithm GRANE satisfies
  \[\|\bx^k-\bx^*\|_{\Fro}^2\le \exp\left\{-\frac{k}{\gamma_{r}^2}\right\},\]
  where $\gamma_{r} = \frac{\LFa}{\mura}$ and $\bx^*$ is the Nash equilibrium matrix.
\end{theorem}

\begin{proof}
According to Theorem~\ref{th:exist1}, there exists a Nash equilibrium $x^*$ in the game $\Gamma(n,\{J_i\},\{\Om_i\},\Gra)$ with compact action sets and satisfying Assumptions~\ref{assum:convex}. Moreover, it can be shown that under Assumption~\ref{assum:restr_str_mon} the Nash equilibrium is unique. Indeed, due to Theorem~\ref{th:VINE}, any Nash equilibrium $x^*\in\Om$ solves $VI(\Om,\bF)$, namely for any $x\in\Om$
\[\langle \bF(x^*), x-x^* \rangle \ge 0.\]
Let us assume there exists another Nash equilibrium $y^*\in\Om$. Then, the inequality above implies
\begin{align*}
  \langle \bF(x^*), y^*-x^* \rangle  &\ge 0, \\
  \langle \bF(y^*), x^*-y^* \rangle  &\ge 0.
\end{align*}
Hence, $\langle \bF(x^*)-\bF(y^*), x^*-y^* \rangle \le 0$, which together with Assumption~\ref{assum:restr_str_mon} implies $x^* = y^*$.

According to the previous discussion, the distributed setting of the GRANE presented in \textbf{Algorithm 1} can be expressed in terms of updates of estimation matrices follows:
\begin{align*}
  \bx^0 &= \bx \in \Oma,\cr
  \bx^{k+1} &= \pro{\Oma}{\bx_k -\lambda \bFa(\bx_k)},\quad k\ge0
\end{align*}
%\[\|\bx^{k+1} - \bx^*\|& \le \|\bx_k -\lambda \bFa(\bx_k) - \bx^*\|. \]
Moreover, since the Nash equilibrium matrix $\bx^*$ solves $VI(\Oma,\bFa)$ (see Proposition~\ref{prop:opt}), we conclude that for any $\lambda>0$
\[\bx^* = \pro{\Oma}{\bx^*_k -\lambda \bFa(\bx^*_k)}\]
Thus, we get the following estimation for the distance to  $\bx^*\in\Oma$:
\begin{align}\label{eq:dist}
  \|\bx^{k+1} - \bx^*\|^2_{\Fro}& \le \|\bx_k -\lambda \bFa(\bx_k) - \bx^* + \lambda \bFa(\bx^*)\|^2_{\Fro}  \cr
  & = \|\bx_k - \bx^*\|^2_{\Fro} \cr
  &\quad+ \lambda^2\| \bFa(\bx_k) - \bFa(\bx^*)\|^2_{\Fro}\cr
  &\quad -2 \lambda\langle\bFa(\bx_k) - \bFa(\bx^*), \bx_k - \bx^*\rangle.
\end{align}
Next, according to Lemma~\ref{lem:augmap_restr} and Remark~\ref{rem:restr}, there exists $\alpha>0$ such that $\bFa$ is restricted strongly monotone in respect to $\bx^*$ with some $\mura>0$ and, hence,
\[\lambda\langle\bFa(\bx_k) - \bFa(\bx^*_k), \bx_k - \bx^*\rangle\ge\mura\|\bx_k - \bx^*\|^2_{\Fro}.\]
Thus, taking into account Lemma~\ref{lem:augmap_LC}, we obtain from \eqref{eq:dist} that
\begin{align*}
  \|\bx^{k+1} - \bx^*\|^2_{\Fro} \le (1 + \LF^2\lambda^2 - 2\mura\lambda)\|\bx_k - \bx^*\|^2_{\Fro}.
\end{align*}
Hence, under the optimal choice of $\lambda = \frac{\mura}{\LFa^2}$, we conclude that
\begin{align*}
  \|\bx^{k+1} - \bx^*\|^2_{\Fro} \le \left(1 - \frac{1}{\gamma_r^2}\right)\|\bx_k - \bx^*\|^2_{\Fro},
\end{align*}
and, thus, the result follows.
\end{proof}

\subsection{Discussion}

In this section, we study how the functional conditions (Lipschitz constants, strong convexity, etc.), the graph topology (algebraic connectivity), and the ratio parameter $\alpha$ affect the convergence speed of the algorithms.

As illustrated in the above results, to reach $\epsilon$-accuracy, GRANE requires $O(\gamma_{r}^2\log(1/\epsilon))$ iterations (Theorem~\ref{th:geom_GRANEr}) in the case of the restricted strongly monotone augmented game mapping $\bFa$ with the constant $\mu_{r,\bFa}$, while Acc-GRANE is applicable in the case of the strongly monotone augmented game mapping $\bFa$ with the constant $\mu_{\bFa}$ and needs $O(\gamma_{\bFa}\left[\log(2F(\by^0))+\log(\gamma_{\bFa}^2/\mu_{\bFa})+\log(1/\epsilon)\right])$ (Theorem~\ref{th:geom_Acc-GRANE}), where we usually consider  $\log(2F(\by^0))$ and $\log(\gamma_{\bFa}^2/\mu_{\bFa})$ to be negligible compared to $\log(1/\epsilon)$. Thus, the major factor that impacts the convergence rates/complexities is $\gamma_{r}$ and $\gamma_{\bFa}$ equal to $\frac{\LFa}{\mura}$ and $\LFa/\muFa$ respectively.

First, let us consider the constant $\gamma_{\bFa}$ defining the convergence rate of Acc-GRANE. We assume that $\Lambda=\alpha I$ and denote $H=\max_i\{L_{-i}\}$. Then the condition number can be expressed as follows:
\begin{equation}\nonumber
\begin{array}{rl}
\gamma_{\bFa}=\max&\left\{\frac{\alpha\LF+\smax{I-W}}{\lminnz{I-W}-0.5\alpha(\sqrt{\muF^2+H^2}-\muF)},\right.\\
&\left.
\frac{\alpha\LF+\smax{I-W}}{\alpha(\muF-H\sqrt{n-1})/n}\right\}.
\end{array}
\end{equation}
By restricting to a smaller class of problem where $H\leq0.5\muF/\sqrt{n-1}$, we are able to obtain that
\begin{equation}
\begin{array}{rl}
\gamma_{\bFa}\leq\max&\left\{\frac{\alpha\LF+\smax{I-W}}{\lminnz{I-W}-\alpha\muF/(16(n-1))},\right.\\
&\left.\frac{\alpha\LF+\smax{I-W}}{\alpha\muF/(2n)}\right\}.
\end{array}\label{eq:final_cond0}
\end{equation}
Note that in Lemma \ref{lem:augmap_strmon}, Remark \ref{rem:augm_str_mon}, and Theorem \ref{th:geom_Acc-GRANE}, the applicable class of games is restricted to those with $H\leq\muF/\sqrt{n-1}$. In the following, we discuss the results under different choices of $\alpha$. Let us denote $C=\frac{16(n-1)\lminnz{I-W}}{\muF}$.

If we choose $\alpha=\frac{C}{9}$, then
\begin{equation}
\begin{array}{c}
\gamma_{\bFa}\leq2n\frac{\LF}{\muF}+\frac{9}{8}\frac{\lmax{I-W}}{\lminnz{I-W}},
\end{array}\label{eq:final_cond}
\end{equation}
where $\gamma_{\bF}\triangleq\LF/\muF$ can be understood as the condition number of the game mapping $\bF$ and $\gamma_{\Gra}\triangleq\lmax{I-W}/\lminnz{I-W}$ is strongly dependent on the algebraic connectivity of the graph. When $W$ is chosen in the form of $W=I-t\Lap$ where $\Lap$ is the Laplacian matrix of the graph $\Gra$ and $t\in(0,2/\lmax{\Lap})$, we have $\gamma_{\Gra}=\lmax{\Lap}/\lminnz{\Lap}$ which coincides with the conventional condition number of the Laplacian. An unpleasant comment for \eqref{eq:final_cond} is that the number of nodes has a linear multiplicative effect on the functional condition number $\gamma_{\bF}$, which implies either the bound is not tight or the algorithm may suffer slow convergence when the number of nodes becomes large. Nonetheless an inspiring observation on \eqref{eq:final_cond} is that the condition number $\gamma_{\bFa}$ is upper bounded by the weighted sum of $\gamma_{\Gra}$ and $\gamma_{\bF}$. This is very different from what many other decentralized algorithms 
illustrate in the literature. For
example, in reference \cite{Shi2014_0} for consensus optimization employing ADMM, the analogous complexity derived under our notation would be $O((\gamma_{\Gra}^2+\gamma_{\Gra}\gamma_{\bF})\log(1/\epsilon))$ where a multiplicative coupling of $\gamma_{\Gra}$ and $\gamma_{\bF}$ is noticed. More examples of such multiplicative coupling can be found in Table II of reference \cite{li2017decentralized}.

By \eqref{eq:final_cond0}, if we choose $\alpha=rC$ where $r\in\left(0,\frac{1}{9}\right)$, then
\begin{equation}
\begin{array}{c}
\gamma_{\bFa}\leq 2n\frac{\LF}{\muF}+\frac{1}{8r}\frac{n}{n-1}\frac{\lmax{I-W}}{\lminnz{I-W}};
\end{array}\label{eq:final_cond2}
\end{equation}
If we choose $\alpha=rC$ where $r\in\left(\frac{1}{9},1\right)$, then
\begin{equation}
\begin{array}{c}
\gamma_{\bFa}\leq\frac{r}{1-r}16n\frac{\LF}{\muF}+\frac{1}{1-r}\frac{\lmax{I-W}}{\lminnz{I-W}}.
\end{array}\label{eq:final_cond3}
\end{equation}
Note that the right-hand-sides of \eqref{eq:final_cond2} and \eqref{eq:final_cond3} are both larger than the right-hand-side of \eqref{eq:final_cond}. We can see that for any absolute constant $r\in(0,1)$, we have $\gamma_{\bFa}=O(n\gamma_{\bF}+\gamma_{\Gra})$. Furthermore, too large $\alpha$ or too small $\alpha$ can both lead to poor scalability and slow convergence while a theoretically optimal $\alpha$ is provided as $\alpha=\frac{C}{9}$.

Similarly to the discussion above, we can obtain the following estimation for the constant $\gamma_{r}$ from Theorem~\ref{th:geom_GRANEr}. 
Indeed, under the choice of the parameters $\beta$ and $\alpha$ as in Remark~\ref{rem:restr}, 
\begin{equation*}
\begin{array}{c}
\gamma_{r}\leq\max\left\{\frac{\alpha\LF+\smax{I-W}}{\alpha\mu_{r,\bFa}/2n},\frac{\alpha\LF+\smax{I-W}}{\lminnz{I-W}}\right\},
\end{array}
\end{equation*}
which implies that
\begin{equation*}
\begin{array}{c}
\gamma_r \le 2n\frac{\LF}{\mu_{r,\bFa}} + \frac{\lmax{I-W}}{\lminnz{I-W}}.
\end{array}
\end{equation*}
Thus, analogously to the analysis of the constant $\gamma_{\bFa}$, we conclude that $\gamma_{r}=O(n\gamma_{r}+\gamma_{\Gra})$ and, hence, the number of players, condition number of the initial game mapping, and connectivity of the communication graph contribute the most to the convergence rate of GRANE.

\section{Simulation Results}\label{sec:sim}
Let us consider a class of games that we have discussed in Remark \ref{rem:augm_str_mon}. Specifically, we have players $\{1,2,\ldots,20\}$ and each player $i$'s objective is to minimize the cost function $J_i(x_i,x_{-i})=f_i(x_i)+l_i(x_{-i})x_i$, where $f_i(x_i)=0.5a_ix_i^2+b_ix_i$ and $l_i(x_{-i})=\sum_{j\neq i}c_{ij}x_j$. Each player $i$ is imposed with a randomly generated box constraint for its own strategy. The local cost function is in general dependent on actions of all players, but the underlying communication graph is a randomly generated tree graph. For simplicity, we choose $\alpha_i=\alpha$ for all $i$.   We randomly select $a_i$, $b_i$, and $c_{ij}$ for all possible $i$ and $j$ but manipulate the data so that $c_{ij}=-c_{ji}$ and $\mu_{\bFa}=\frac{\alpha}{n}(\min_i\{a_i\}-((n-1)\max_i\{\sum_{j\neq i}c_{ij}^2\})^{0.5})>0$. Thus, according to Theorem~\ref{th:geom_Acc-GRANE}, the accelarated gradient algorithm, namely Acc-GRANE, can be applied to learn the Nash equilibrium in the game under 
consideration. Convergence curves for GRANE and Acc-GRANE are provided in Fig.~\ref{eps:Conv}. The Nash
equilibrium $x^*$ is computed using centralized gradient play method with over $12000$ iterations. We use this $x^*$ as our reference and define the relative error as $\|\bx^k-\bx^*\|_\Fro^2/\|\bx^k-\bx^0\|_\Fro^2$. The convergence of GRANE is slow due to the small step-size and the small parameter $\alpha$. 

The small parameter $\alpha$ and specific settings for the coefficients $a_i$, $c_{ij}$, $i,j=1,\ldots, 20$, are chosen to guarantee that $\mu_{\bFa}>0$ and, thus, Acc-GRANE is applicable. To relax these assumptions on the cost functions and choose a larger $\alpha$, we refer to the results in Lemma~\ref{lem:augmap_restr} and Theorem~\ref{th:geom_GRANEr}.  These results allow us to choose $\alpha$ as in Remark~\ref{rem:restr} and to apply GRANE to the game with any $a_i$, $c_{ij}$, $i,j=1,\ldots, 20$, such that $\min_ia_i>0$. The implementation of GRANE for different settings of the cost functions and the corresponding parameter $\alpha$ are demonstrated in Fig.~\ref{eps:Conv1}. As we can see, in 
contrast to the run of GRANE applied to the strongly monotone augmented mapping $\bFa$ (see Fig.~\ref{eps:Conv}, where the relative error does not change significantly after $500$ iterations), the relative error of GRANE applied to the restricted strongly monotone augmented mapping $\bFa$ decreases with time.
%  \vspace{0.31cm}
% \begin{figure}[ht]
% 	\begin{center}
% 		\includegraphics
% 		[width=1\linewidth,trim={1cm 7cm 1cm 8cm},clip]{Conv.eps}\vspace{0em}
% 		\caption{Plots of the normalized residual $\|\bx^k-\bx^*\|_\Fro/\|\bx^0-\bx^*\|_\Fro$. Different $\alpha$ is set to observe the performance under different $1/\LFa$.\label{eps:Conv}}
% 	\end{center}
% \end{figure}
% 
 \vspace{0.31cm}
\begin{figure}[h]
	\begin{center}
		\includegraphics
		[width=1\linewidth]{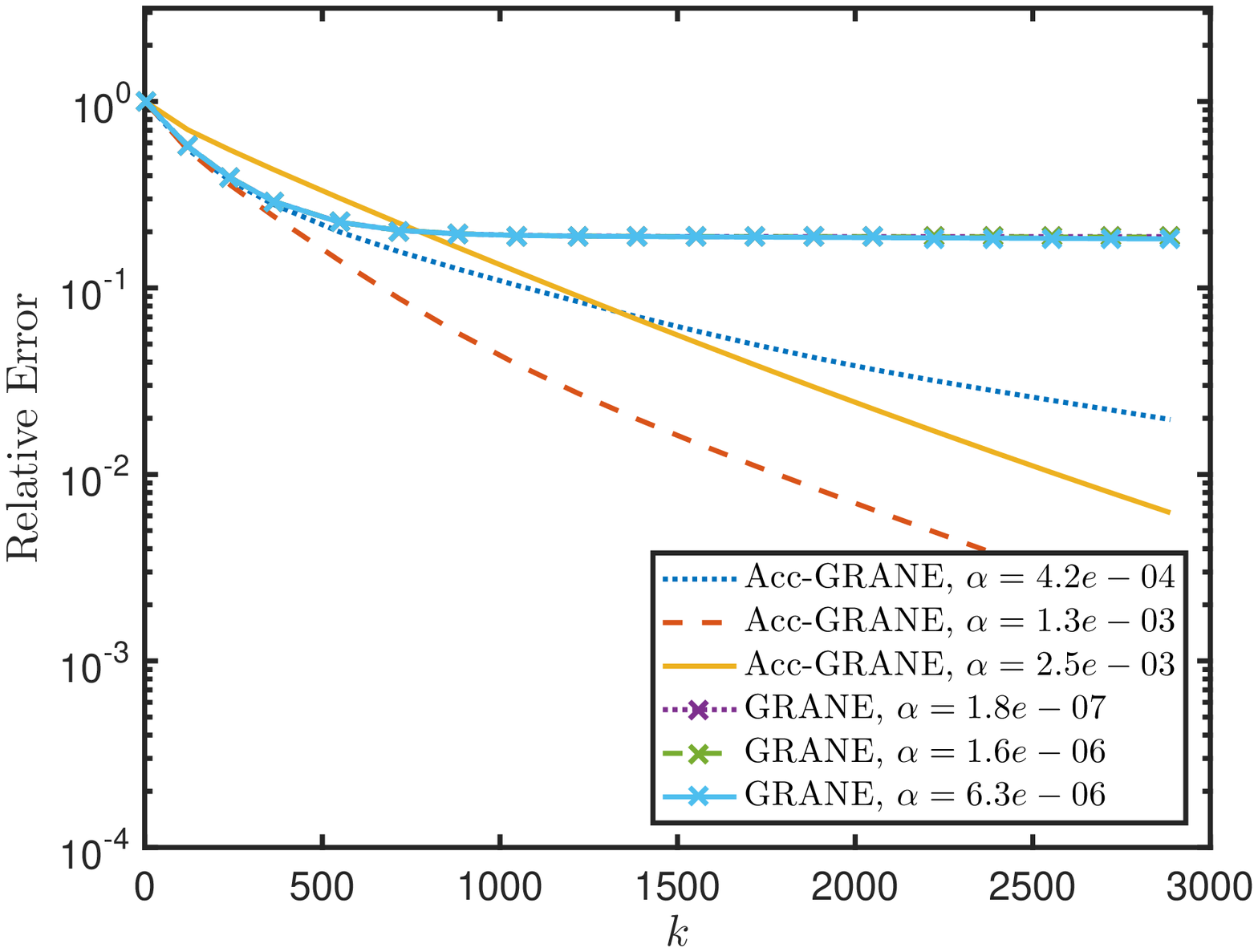}\vspace{-1em}
		\caption{Plots of the normalized residual $\|\bx^k-\bx^*\|_\Fro/\|\bx^0-\bx^*\|_\Fro$. Different $\alpha$ is set to observe the performance under different $1/\LFa$.\label{eps:Conv}}
	\end{center}
\end{figure}

\begin{figure}[ht]
	\centering
	\psfrag{0}[c][l]{\small$0$}
	\psfrag{0.5}[c][c]{\small$0.5$}
	\psfrag{0.1}[c][c]{\small$0.1$}
	\psfrag{0.01}[c][c]{\small$0.01$}
	\psfrag{k}[c][l]{\small$k$}
	\psfrag{2000}[c][c]{\small$2000$}
	\psfrag{4000}[c][c]{\small$4000$}
	\psfrag{6000}[c][c]{\small$6000$}
	\psfrag{8000}[c][c]{\small$8000$}
	\psfrag{10000}[c][c]{\small$10000$}
	\psfrag{RE}[c][l]{\small{Relative Error}}
	\psfrag{GRANE}[c][c]{GRANE}
        \includegraphics[width=1\linewidth]{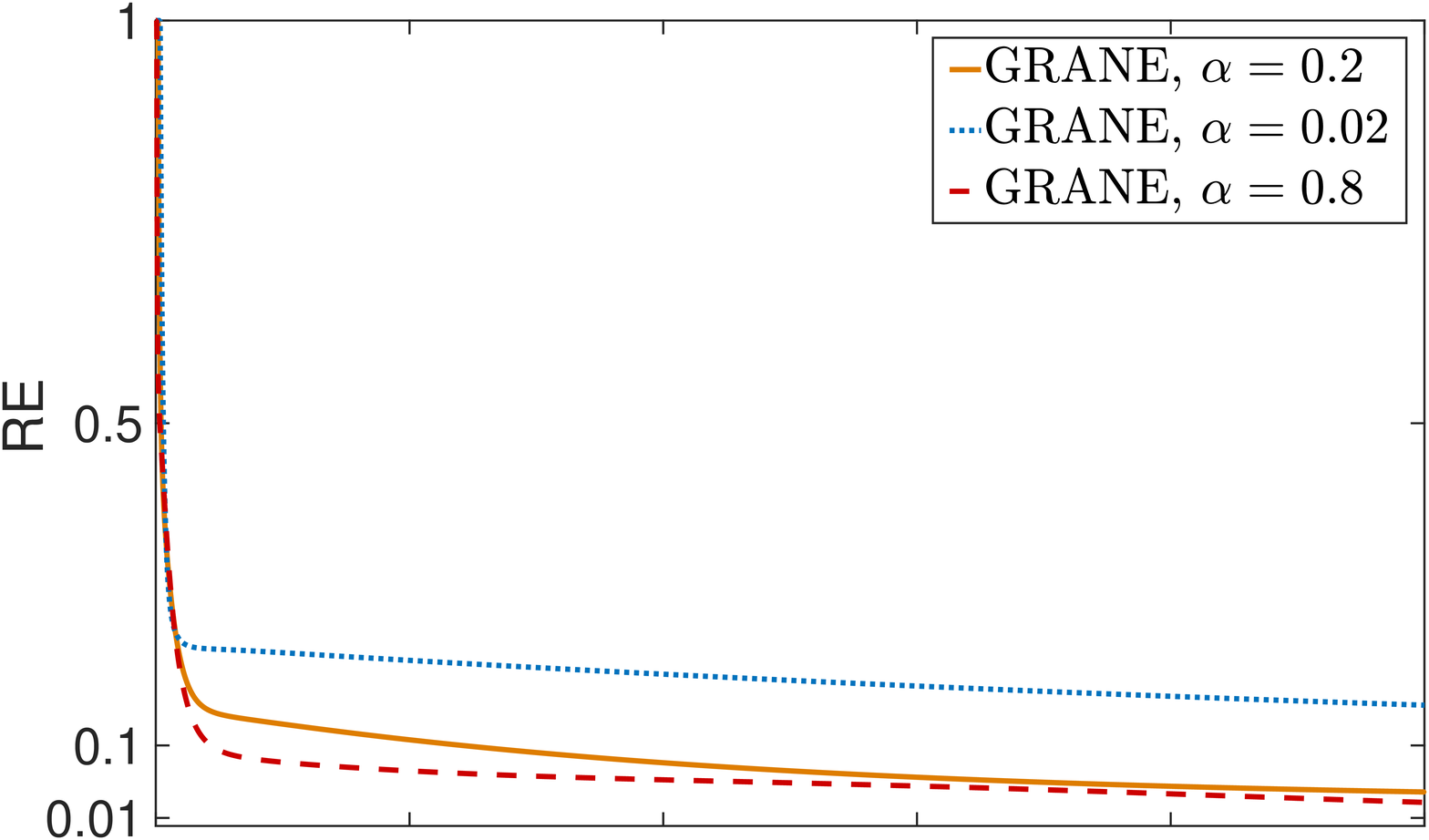}
        \caption{Plots of the normalized residual $\|\bx^k-\bx^*\|_\Fro/\|\bx^0-\bx^*\|_\Fro$. Different $\alpha$ is set to observe the performance under different $1/\LFa$.\label{eps:Conv1}}
        \end{figure}

\section{Conclusion}\label{sec:conclusion}
In this paper, we have studied a class of games with (restricted) strongly monotone mappings. We have considered the case in which players do not have the full information of joint actions but can exchange their estimates with neighbors through a decentralized communication network. To let the networked system reach the Nash equilibrium at a geometric rate, a decentralized gradient play algorithm is introduced based on the investigation of the conditions for a Nash equilibria. Moreover, by leveraging the idea of an accelerated approach for solving variational inequalities, we have obtained a version of the gradient play algorithm that guaratees a faster convergence to the Nash equilibrium, due to a better dependence on the condition number.
The presented accelerated algorithm is applicable to a restricted class of games for which the augmented game mapping is strongly monotone.
To apply the decentralized gradient play algorithm to a broader class of games, we consider the case of the restricted strongly monotone augmented game mapping and demonstrated geometric convergence of the procedure to the Nash equilibrium.
Future works can be devoted to the investigation of modifications for the accelerated algorithm to make it applicable in games without assumption on strongly monotone augmented game mappings. Another future direction is to adapt the procedures to settings that do not require knowledge of the strong monotonicity constant and the Lipschitz continuity constant of the game mapping.

\bibliographystyle{plain}
\bibliography{document}

\end{document}